\documentclass[11pt,a4paper,onesided]{article}

\usepackage{ amssymb, amsmath, amsthm, graphics, xspace, enumerate}
\usepackage{graphicx}
\usepackage[a4paper,colorlinks=true,citecolor=blue,urlcolor=blue,linkcolor=blue,bookmarksopen=true,unicode=true,pdffitwindow=true]{hyperref}
 \hypersetup{pdfauthor={Guodong Pang and Nikola Sandri\'{c}}}
\hypersetup{pdftitle={Ergodicity and fluctuations of a fluid particle driven by
diffusions with jumps}}

\newtheorem{tm}{tm}[section]
\newtheorem{theorem}[tm]{Theorem}

\newtheorem{proposition}[tm]{Proposition}

\newtheorem{definition}[tm]{Definition}

\newtheorem{remark}[tm]{Remark}

\newcommand{\process}[1]{\{#1_t\}_{t\geq0}}
\newcommand{\chain}[1]{\{#1_n\}_{n\geq0}}
\newcommand {\R} {\ensuremath{\mathbb{R}}}
\newcommand {\ZZ} {\ensuremath{\mathbb{Z}}}
\newcommand {\N} {\ensuremath{\mathbb{N}}}
\newcommand {\CC} {\ensuremath{\mathbb{C}}}
\newcommand {\QQ} {\ensuremath{\mathbb{Q}}}
\numberwithin{equation}{section}
\def\be{\begin{equation}}
\def\ee{\end{equation}}
\begin{document}

 \title{Ergodicity and Fluctuations of a Fluid Particle Driven by Diffusions with Jumps}
 \author{ Guodong Pang\\ The Harold and Inge Marcus Department of Industrial and Manufacturing
 Engineering\\
 College
of Engineering, Pennsylvania State University, University
Park PA 16802, USA\\ \bigskip
Email: gup3@psu.edu\\
Nikola Sandri\'{c}\\ Institut f\"ur Mathematische Stochastik\\ Fachrichtung Mathematik, Technische Universit\"at Dresden, 01062 Dresden, Germany\\
and\\
Department of Mathematics\\
         Faculty of Civil Engineering, University of Zagreb, 10000 Zagreb,
         Croatia \\
        Email: nsandric@grad.hr}

 \maketitle
\begin{center}
{
\medskip

} \end{center}

\begin{abstract}
In this paper, we study the long-time behavior of a fluid particle immersed in a turbulent fluid  driven by a diffusion with jumps, that is, a Feller process associated with a
non-local operator.  We derive the law of large numbers and  central limit theorem for the  evolution process of the tracked fluid particle in the cases when the driving  process: (i) has periodic coefficients, (ii) is ergodic or (iii)  is a class of L\'evy processes.
 The presented results generalize the  classical and well-known results for fluid flows driven by elliptic diffusion processes.

\end{abstract}
\smallskip


%
%
%
%


\allowdisplaybreaks

\section{Introduction}\label{s1}

\ \ \ \ Turbulence is one of the most important phenomena in nature
and engineering. It is a flow regime characterized by the presence
of irregular eddying motions, that is, motions with high level (high
Reynolds number) of vorticity. The key problem is to describe the
chaotic motion of a turbulent fluid. In practice this is done by
tracking the evolution of a specially marked physical entity
(particle) which is immersed in the
 fluid. Clearly, such a particle  must be
light and small enough (noninertial) so that its presence does not
affect the flow pattern. In this way, the motion of the  fluid may
be visualized through the evolution of this  passively advected
particle which follows the streamlines of the fluid. In experimental
sciences such a particle is often called a \emph{fluid particle} or
\emph{passive tracer}. The evolution of a fluid particle is
described by the following transport equation \be\label{eq1.1}
\dot{\rm{\mathbf{x}}}(t)={\rm \mathbf{v}}(t,{\rm \mathbf{x}}(t)),\quad {\rm \mathbf{x}}(0)=x_0,\ee where ${\rm \mathbf{v}}(t,x)\in\R^{d}$,
$d\geq1$, is the turbulent  velocity vector field which describes
the movement of the fluid at point $x\in\R^{d}$ in space at time
$t\geq0$ and ${\rm \mathbf{x}}(t)\in\R^{d}$ is the position of the particle at time
$t\geq0$. However, as we mentioned, turbulence is a chaotic process.
More precisely,
the velocity vector field  of any fluid flow which advects the fluid
particle should be a solution to the Navier-Stokes equation. But,
solutions to the Navier-Stokes equations for very turbulent  fluids
are unstable, that is, they have sensitive dependence on the initial
conditions that makes the fluid flow irregular both in space and
time. In other words, the velocity field ${\rm \mathbf{v}}(t,x)$ at a fixed point
varies with time in a nearly random manner. Similarly, ${\rm \mathbf{v}}(t,x)$ at a
fixed time varies with position in a nearly random manner. Due to this, a probabilistic approach to
this problem might be adequate and it might bring a substantial understanding of turbulence. Accordingly, our aim
is to study certain statistical properties of the turbulence through
simplified random velocity field models which possess some empirical
properties of turbulent fluid flows. Based on the symmetries of the
Navier-Stokes equation, it is well known that random velocity fields
 of  very turbulent flows (with high Reynolds numbers), among other properties, are
 time stationary, space homogeneous and isotropic. We refer the reader to
 \cite{chorin} and \cite{frisch}  for an
extensive overview on turbulent  flows.

 Now, instead of the transport equation (\ref{eq1.1}) we consider  the
following It\^o's stochastic differential equation \be\label{eq1.2}
dX_t=V(t,X_t)dt+ dB_t,\quad X_0=x_0.\ee Here, $\{V(t,x)\}_{t\geq0,\, x\in\R^{d}}$ is a
$d$-dimensional, $d\geq1$, random   velocity vector field defined on
a probability space $(\Omega_{V},\mathfrak{F}_{V},\mathbb{P}_{V})$ which
describes the movement of a turbulent fluid  and $\process{B}$
is a $d$-dimensional zero-drift Brownian motion  defined on a probability space
$(\Omega_B,\mathfrak{F}_B,\mathbb{P}_B)$ and given by a covariance matrix $\Sigma=(\sigma_{ij})_{i,j=1,\ldots,d}$ describing the molecular diffusivity of
the fluid. Recall that if for any
$T>0$ there exist random constants $C_T$ and $D_T$  such that
$$\sup_{0\leq t\leq T}|V(t,x,\omega_V)-V(t,y,\omega_V)|<C_T(\omega_V)|x-y|,\quad x,y\in\R^{d},\ \mathbb{P}_V\textrm{-a.s},$$ and
$$\sup_{0\leq t\leq T}|V(t,x,\omega_V)|<D_T(\omega_V)(1+|x|),\quad x\in\R^{d},\ \mathbb{P}_V\textrm{-a.s},$$ then (\ref{eq1.2}) has a unique solution (see \cite[Theorem 5.2.1]{oksendal}) defined on $(\Omega_V\times\Omega_B,\mathfrak{F}_V\times\mathfrak{F}_B,\mathbb{P}_V\times\mathbb{P}_B)$ which can be seen as an elliptic diffusion process in a turbulent random  environment.

The main goal is to describe certain statistical
properties of the fluid flow (that is, of $\process{X}$) through the
statistical properties of the velocity  field $\{V(t,x)\}_{t\geq0,\, x\in\R^{d}}$. In
particular, we are interested in the long-time behavior of $\process{X}$
(clearly, for small times $X_t\approx x_0$). More precisely, we
investigate whether
\be\label{eq1.3}\frac{X_{nt}}{n}\stackrel{\hbox{\scriptsize{$\mathbb{P}_{V}\times\mathbb{P}_B\, \textrm{-}\, \textrm{a.s.}$}}}{\xrightarrow{\hspace*{1.8cm}}}\bar{V} t\ee
as $n\longrightarrow\infty$, for some $\bar{V} \in\R^{d},$ and, if this is the
case, we analyze fluctuations of $\process{X}$ around $\bar{V}$, that is,
we investigate whether
\be\label{eq1.4}\left\{n^{\frac{1}{2}}\left(\frac{X_{nt}}{n}-\bar{V} t\right)\right\}_{t\geq0}\stackrel{\hbox{\scriptsize{$\textrm{d}$}}}{\longrightarrow}\process{W}\ee
as $n\longrightarrow\infty$. Here,
$\stackrel{\hbox{\scriptsize{$\textrm{d}$}}}{\longrightarrow}$
denotes the convergence in distribution, in the space of c\'adl\'ag
functions endowed with the Skorohod $J_1$ topology (see  \cite{billingsley-book} or \cite{jacod} for details), and
$\process{W}$ is
 a $d$-dimensional (possibly degenerated) zero-drift Brownian motion.

Long-time behavior of $\process{X}$ of type (\ref{eq1.3}) and
(\ref{eq1.4}) has been very extensively studied in the literature.
In particular, ergodicity of $\process{X}$, under the
assumption that $\{V(t,x)\}_{t\geq0,\, x\in\R^{d}}$ is regular enough and has only finite
dependency  in time or space,  has been deduced in \cite{kom-krupa-inv1} and \cite{kom-krupa-inv2}. Regarding the
analysis of fluctuations of $\process{X}$,  yet in 1923 G. I. Taylor
\cite{taylor} noticed that if the velocity field $\{V(t,x)\}_{t\geq0,\, x\in\R^{d}}$
decorrelates sufficiently fast in time or space, then the limit in
(\ref{eq1.4}) should have a diffusive character. A rigorous
mathematical analysis and  proofs of this fact have occupied many
authors. By assuming certain additional structural and statistical
properties of the velocity field $\{V(t,x)\}_{t\geq0,\, x\in\R^{d}}$ (time or space
independence, Markovian or Gaussian nature, strong mixing property
in time or space), Taylor's observation has been confirmed (see
\cite{kom-fan-1997}, \cite{kom-fan-1999},  \cite{kom-fan-2001}, \cite{kom-fan-2002},  \cite{fan-ryz-pap},
 \cite{pap-fan},
 \cite{pap-kes}, \cite{kom-ola-2001}, \cite{maj-kra}, \cite{pap-str-var}, \cite{pap-var}     and the reference
therein). Also, let us remark that  the lack of  long-range
(temporal or spatial) decorrelations of $\{V(t,x)\}_{t\geq0,\, x\in\R^{d}}$ may lead to memory
effects, that is, an ``anomalous" (non-Markovian) diffusive behavior
(fractional Brownian motion, local times of certain Markov
processes, subordinated Brownian motion,
exponential random variable) may appear as a limit in (\ref{eq1.4})
(see \cite{fan, fan-er},  \cite{kom-fan-2000-1},
\cite{kom-fan-2000-2}, \cite{kom-fan-2003}, \cite{kom-nov-ryz},  \cite{nua-xu}, \cite{pap-str-var} and the
references therein).

In this paper,  we consider a model in which the velocity field
$\{V(t,x)\}_{t\geq0,\, x\in\R^{d}}$ is space independent and its time dependence and randomness
are governed by a diffusion with jumps. More precisely,
$$V(t,x,\omega_V):=v(F_t(\omega_V)),\quad t\geq0,\ \omega_V\in\Omega_V,$$
where $v:\R^{\bar{d}}\longrightarrow\R^{d}$, $\bar{d}\geq1$, is a certain function (specified below) and $\process{F}$ is an $\R^{\bar{d}}$-valued diffusion with jumps (Feller process) determined by an
integro-differential operator (infinitesimal generator)
$(\mathcal{A},\mathcal{D}_\mathcal{A})$ of the form
\begin{align}\label{eq1.5}\mathcal{A}f(x)&=\langle b(x),\nabla f(x)\rangle+
\frac{1}{2}\rm{div}\it{c}(x)\nabla f(x)\nonumber
\\&\ \ \ +\int_{\R^{\bar{d}}}\left(f(y+x)-f(x)-\langle y,\nabla f(x)\rangle1_{\{z:|z|\leq1\}}(y)\right)\nu(x,dy),\quad
f\in\mathcal{D}_\mathcal{A}.\end{align} Our work is highly motivated
by  the works of A. Bensoussan, J-L. Lions and G. C. Papanicolaou \cite{benso-lions-book}, R. N.
Bhattacharya \cite{bhat} and G. C. Papanicolaou, D. Stroock and S. R. S.
Varadhan \cite{pap-str-var} in which they consider a model with
$\process{F}$ being a diffusion process determined by a second-order
elliptic operator $(\mathcal{A},\mathcal{D}_\mathcal{A})$ of the
form
$$\mathcal{A}f(x)=\langle b(x),\nabla f(x)\rangle+
\frac{1}{2}\rm{div}\it{c}(x)\nabla f(x),\quad
f\in\mathcal{D}_\mathcal{A},$$ and, under the  assumptions
that either $\process{F}$ is a diffusion  on the
$\bar{d}$-dimensional torus $\R^{\bar{d}}/\ZZ^{\bar{d}}$ and
$v(x)=(\mathcal{A}w_1(x),\ldots,\mathcal{A}w_{d}(x)),$ for
$w_1,\ldots,w_{d}\in C^{2}(\R^{\bar{d}}/\ZZ^{\bar{d}})$, or
$\process{F}$ is ergodic and
$v(x)=(\mathcal{A}w_1(x),\ldots,\mathcal{A}w_{d}(x))$, for
$w_1,\ldots,w_{d}\in C^{\infty}_c(\R^{\bar{d}}),$ they derive the
Brownian limit in  (\ref{eq1.4}).
Here
we extend their results by investigating the
long-time behaviors in (\ref{eq1.3}) and (\ref{eq1.4}) of $\process{X}$
driven by a diffusion with jumps.

We have identified three sets of conditions for the driving diffusion with jumps $\process{F}$ under which the law of large numbers (LLN) in \eqref{eq1.3} and the central limit theorem (CLT) in \eqref{eq1.4} hold for the process $\process{X}$.
In the first case, in Theorem \ref{tm1.1},  the driving process $\process{F}$ has  periodic coefficients $(b(x),c(x),\nu(x,dy))$.  Here we also  assume the existence, continuity (in space variables) and strict positivity of a transition density function $p(t,x,y)$ of $\{F_t\}_{t\ge 0}$; see discussions on the assumption in Remark \ref{rm2}.  These assumptions imply implicitly that the projection of the driving diffusion with jumps on the torus $\R^{\bar{d}}/\ZZ^{\bar{d}}$ is  ergodic.  In Theorem \ref{tm1.2}, we simply assume that the driving diffusion with jumps is ergodic, and establish  the limiting properties in \eqref{eq1.3} and \eqref{eq1.4}.
Since  any L\'evy process has constant coefficients  (L\'evy triplet), in Theorem \ref{tm1.3} we establish
the limiting properties in \eqref{eq1.3} and \eqref{eq1.4} for a class of L\'evy processes with certain coefficients properties which relax the conditions from Theorem \ref{tm1.1}.
 Note that (non-trivial) L\'evy processes are never ergodic, hence, in the L\'evy process case, the results in Theorem  \ref{tm1.2} do not apply. We also discuss the cases when the driving diffusions with jumps  are not necessarily ergodic in Section \ref{sec-Discussions}.

The main techniques used  in  \cite{benso-lions-book},  \cite{bhat} and \cite{pap-str-var}  are based on proving the  convergence of finite-dimensional distributions of the underlying processes, functional central limit theorems for stationary ergodic sequences and solving
 martingale problems, respectively. On the other hand, our approach in proving the main results, Theorems \ref{tm1.1}, \ref{tm1.2} and \ref{tm1.3},
 is through the  characteristics of a semimartingale (note that the process $\process{X}$ in our setting is a semimartingale). More precisely,
 by using the facts that $\process{F}$ and $\process{B}$ are independent and the regularity assumptions imposed upon $\process{F}$, and by the classical Birkhoff ergodic theorem, we can conclude the limiting behavior in (\ref{eq1.3}). To obtain the Brownian limit
  in (\ref{eq1.4}), we reduce the problem to the convergence of the corresponding  semimartingale characteristics. Namely,
   since $\process{F}$ is a semimartingale whose characteristics is given in terms of its L\'evy triplet
   (see \cite[Lemma 3.1 and Theorem 3.5]{rene-holder}), we explicitly compute the characteristics of $\process{X}$
   and show that it converges (in probability) to the characteristics of the Brownian motion $\process{W}$, which,
   according to \cite[Theorem VIII.2.17]{jacod}, proves the desired results.

The sequel of this paper is organized as follows.
In Section \ref{sec-DiffJumps}, we give some preliminaries on diffusions with jumps. In Section \ref{sec-MainResults}, we state the main results of the paper, Theorems \ref{tm1.1}, \ref{tm1.2} and \ref{tm1.3}. In Section \ref{sec-Proofs1}, we prove Theorems \ref{tm1.1} and \ref{tm1.2}, and in Section \ref{sec-Proofs2}, we prove Theorem \ref{tm1.3}. Finally, in Section \ref{sec-Discussions}, we present some discussions on the  ergodicity property of general diffusions with jumps and the limiting behaviors  in  (\ref{eq1.3}) and (\ref{eq1.4}) when the velocity field $\{V(t,x)\}_{t\geq0,\, x\in\R^{\bar{d}}}$ is governed by  general, not necessarily ergodic, diffusions with jumps.

\section{Preliminaries on Diffusions with Jumps}
\label{sec-DiffJumps}

\ \ \ \ Let
$(\Omega,\mathcal{F},\{\mathbb{P}^{x}\}_{x\in\R^{\bar{d}}},$
$\process{\mathcal{F}},\process{\theta},  \process{M})$, denoted by $\process{M}$
in the sequel, be a Markov process with  state space
$(\R^{\bar{d}},\mathcal{B}(\R^{\bar{d}}))$, where $\bar{d}\geq1$ and
$\mathcal{B}(\R^{\bar{d}})$ denotes the Borel $\sigma$-algebra on
$\R^{\bar{d}}$. A family of linear operators $\process{P}$ on
$B_b(\R^{\bar{d}})$ (the space of bounded and Borel measurable functions),
defined by $$P_tf(x):= \mathbb{E}^{x}[f(M_t)],\quad t\geq0,\
x\in\R^{\bar{d}},\ f\in B_b(\R^{\bar{d}}),$$ is associated with the process
$\process{M}$. Since $\process{M}$ is a Markov process, the family
$\process{P}$ forms a \emph{semigroup} of linear operators on the
Banach space $(B_b(\R^{\bar{d}}),||\cdot||_\infty)$, that is, $P_s\circ
P_t=P_{s+t}$ and $P_0=I$ for all $s,t\geq0$. Here,
$||\cdot||_\infty$ denotes the supremum norm on the space
$B_b(\R^{\bar{d}})$. Moreover, the semigroup $\process{P}$ is
\emph{contractive}, that is, $||P_tf||_{\infty}\leq||f||_{\infty}$
for all $t\geq0$ and all $f\in B_b(\R^{\bar{d}})$, and \emph{positivity
preserving}, that is, $P_tf\geq 0$ for all $t\geq0$ and all $f\in
B_b(\R^{\bar{d}})$ satisfying $f\geq0$. The \emph{infinitesimal generator}
$(\mathcal{A}^{b},\mathcal{D}_{\mathcal{A}^{b}})$ of the semigroup
$\process{P}$ (or of the process $\process{M}$) is a linear operator
$\mathcal{A}^{b}:\mathcal{D}_{\mathcal{A}^{b}}\longrightarrow B_b(\R^{\bar{d}})$
defined by
$$\mathcal{A}^{b}f:=
  \lim_{t\longrightarrow0}\frac{P_tf-f}{t},\quad f\in\mathcal{D}_{\mathcal{A}^{b}}:=\left\{f\in B_b(\R^{\bar{d}}):
\lim_{t\longrightarrow0}\frac{P_t f-f}{t} \ \textrm{exists in}\
||\cdot||_\infty\right\}.
$$ We call $(\mathcal{A}^{b},\mathcal{D}_{\mathcal{A}^{b}})$ the \emph{$B_b$-generator} for short.

A Markov process $\process{M}$ is said to be a \emph{Feller process}
if its corresponding  semigroup $\process{P}$ forms a \emph{Feller
semigroup}. This means that the family $\process{P}$ is a semigroup
of linear operators on the Banach space
$(C_\infty(\R^{d}),||\cdot||_{\infty})$  and it is \emph{strongly
continuous}, that is,
  $$\lim_{t\longrightarrow0}||P_tf-f||_{\infty}=0,\quad f\in
  C_\infty(\R^{\bar{d}}).$$ Here, $C_\infty(\R^{\bar{d}})$ denotes
the space of continuous functions vanishing at infinity.
Every Feller semigroup $\process{P}$  can be uniquely extended to
$B_b(\R^{\bar{d}})$ (see \cite[Section 3]{rene-conserv}). For notational
simplicity, we denote this extension again by $\process{P}$. Also,
let us remark that every Feller process possesses the strong Markov
property and has c\`adl\`ag sample paths (see  \cite[Theorems 3.4.19 and
3.5.14]{jacobIII}). This entails that $\process{M}$ is
\emph{progressively measurable}, that is, for each $t>0$ the
function $(s, \omega)\longmapsto M_s(\omega)$ on $[0,
t]\times\R^{\bar{d}}$ is measurable with respect to the $\sigma$-algebra
$\mathcal{B}([0,t])\times\mathcal{F}_t$, where $\mathcal{B}([0, t])$
is the Borel $\sigma$-algebra on $[0, t]$ (see \cite[Proposition
3.6.2]{jacobIII}). In particular, under an appropriate choice of the
velocity function $v(x)$, the process $\process{X}$ in \eqref{eq1.2} is well defined.
Further,
in the case of Feller processes, we call
$(\mathcal{A}^{\infty},\mathcal{D}_{\mathcal{A}^{\infty}}):=(\mathcal{A}^{b},\mathcal{D}_{\mathcal{A}^{b}}\cap
C_\infty(\R^{\bar{d}}))$ the \emph{Feller generator} for short. Note
that, in this case, $\mathcal{D}_{\mathcal{A}^{\infty}}\subseteq
C_\infty(\R^{\bar{d}})$ and
$\mathcal{A}^{\infty}(\mathcal{D}_{\mathcal{A}^{\infty}})\subseteq
C_\infty(\R^{\bar{d}})$.
If the set of smooth functions
with compact support $C_c^{\infty}(\R^{\bar{d}})$ is contained in
$\mathcal{D}_{\mathcal{A}^{\infty}}$, that is, if the Feller
generator
$(\mathcal{A}^{\infty},\mathcal{D}_{\mathcal{A}^{\infty}})$ of the
Feller process $\process{M}$ satisfies
    \begin{description}
      \item[(\textbf{C1})]
      $C_c^{\infty}(\R^{\bar{d}})\subseteq\mathcal{D}_{\mathcal{A}^{\infty}}$,
    \end{description}
 then, according to \cite[Theorem 3.4]{courrege-symbol},
$\mathcal{A}^{\infty}|_{C_c^{\infty}(\R^{\bar{d}})}$ is a \emph{pseudo-differential
operator}, that is, it can be written in the form
\begin{align}\label{eq1.8}\mathcal{A}^{\infty}|_{C_c^{\infty}(\R^{\bar{d}})}f(x) = -\int_{\R^{\bar{d}}}q(x,\xi)e^{i\langle \xi,x\rangle}
\hat{f}(\xi) d\xi,\end{align}  where $\hat{f}(\xi):=
(2\pi)^{-\bar{d}} \int_{\R^{\bar{d}}} e^{-i\langle\xi,x\rangle} f(x) dx$ denotes
the Fourier transform of the function $f(x)$. The function $q :
\R^{\bar{d}}\times \R^{\bar{d}}\longrightarrow \CC$ is called  the \emph{symbol}
of the pseudo-differential operator. It is measurable and locally
bounded in $(x,\xi)$ and continuous and negative definite as a
function of $\xi$. Hence, by \cite[Theorem 3.7.7]{jacobI}, the
function $\xi\longmapsto q(x,\xi)$ has, for each $x\in\R^{\bar{d}}$, the
following L\'{e}vy-Khintchine representation \begin{align}\label{eq1.9}q(x,\xi) =a(x)-
i\langle \xi,b(x)\rangle + \frac{1}{2}\langle\xi,c(x)\xi\rangle -
\int_{\R^{\bar{d}}}\left(e^{i\langle\xi,y\rangle}-1-i\langle\xi,y\rangle1_{\{z:|z|\leq1\}}(y)\right)\nu(x,dy),\end{align}
where $a(x)$ is a nonnegative Borel measurable function, $b(x)$ is
an $\R^{\bar{d}}$-valued Borel measurable function,
$c(x):=(c_{ij}(x))_{1\leq i,j\leq \bar{d}}$ is a symmetric nonnegative
definite $\bar{d}\times \bar{d}$ matrix-valued Borel measurable function
 and $\nu(x,dy)$ is a Borel kernel on $\R^{\bar{d}}\times
\mathcal{B}(\R^{\bar{d}})$, called the \emph{L\'evy measure}, satisfying
$$\nu(x,\{0\})=0\quad \textrm{and} \quad \int_{\R^{\bar{d}}}\min\{1,
|y|^{2}\}\nu(x,dy)<\infty,\quad x\in\R^{\bar{d}}.$$ The quadruple
$(a(x),b(x),c(x),\nu(x,dy))$ is called the \emph{L\'{e}vy quadruple}
of the pseudo-differential operator
$\mathcal{A}^{\infty}|_{C_c^{\infty}(\R^{\bar{d}})}$ (or of the symbol $q(x,\xi)$).
In the sequel, we assume the following conditions on the symbol
$q(x,\xi)$:
\begin{description}
  \item[(\textbf{C2})] $||q(\cdot,\xi)||_\infty\leq c(1+|\xi|^{2})$ for some
  $c\geq0$ and all $\xi\in\R^{\bar{d}}$ ;
  \item[(\textbf{C3})] $q(x,0)=a(x)=0$ for all $x\in\R^{\bar{d}}.$
\end{description}
Let us remark that, according to \cite[Lemma 2.1]{rene-holder},
condition (\textbf{C2}) is equivalent with the boundedness of the
coefficients of the symbol $q(x,\xi)$, that is,
$$||a||_\infty+||b||_\infty+||c||_{\infty}+\left|\left|\int_{\R^{\bar{d}}}\min\{1,y^{2}\}\nu(\cdot,dy)\right|\right|_{\infty}<\infty,$$
and, according to \cite[Theorem 5.2]{rene-conserv}, condition
(\textbf{C3}) (together with condition (\textbf{C2})) is equivalent
with the conservativeness property of the process $\process{M}$,
that is, $\mathbb{P}^{x}(M_t\in\R^{\bar{d}})=1$ for all $t\geq0$ and all
$x\in\R^{\bar{d}}$. Also, by combining (\ref{eq1.8}), (\ref{eq1.9}) and (\textbf{C3}), it is easy to see that $\mathcal{A}^{\infty}$, on $C_c^{\infty}(\R^{\bar{d}})$, has a representation as an integro-differential operator given in (\ref{eq1.5}).

Throughout this paper,  $\process{F}$ denotes a Feller
process satisfying conditions (\textbf{C1}), (\textbf{C2}) and (\textbf{C3}). Such a
process is called a \emph{diffusion with jumps}.
If $\nu(x,dy)=0$ for all $x\in\R^{\bar{d}}$, then $\process{F}$ is just called a \emph{diffusion}.  Note that this definition agrees with the standard definition of elliptic diffusion processes (see \cite{rogersI}).

Further, note that in the case when the symbol $q(x,\xi)$ does not depend
on the variable $x\in\R^{\bar{d}}$, $\process{M}$ becomes a \emph{L\'evy
process}, that is, a stochastic process   with stationary and
independent increments and c\`adl\`ag sample paths. Moreover, unlike Feller
processes, every L\'evy process is uniquely and completely
characterized through its corresponding symbol (see \cite[Theorems
7.10 and 8.1]{sato-book}). According to this, it is not hard to
check that every L\'evy process satisfies conditions
(\textbf{C1}), (\textbf{C2}) and (\textbf{C3}) (see \cite[Theorem 31.5]{sato-book}).
Thus, the class of processes we consider in this paper contains the
class of L\'evy processes.
Also,
 a L\'evy process is denoted by  $\process{L}$. For more
on diffusions with jumps  we refer the readers to the monograph
\cite{rene-bjorn-jian}.

Finally, we recall relevant  definitions of ergodicity of Makov processes.
 Let
$(\Omega,\mathcal{F},\{\mathbb{P}^{x}\}_{x\in
S},$ $\process{\mathcal{F}},\process{\theta},\process{M})$,
denoted by $\process{M}$ in the sequel, be a Markov process on a  state space
$(S,\mathcal{S})$. Here, $S$ is a nonempty set and $\mathcal{S}$ is
a $\sigma$-algebra of subsets of $S$. A probability measure $\pi(dx)$ on
$\mathcal{S}$ is called \emph{invariant} for $\process{M}$ if
$$\int_S\mathbb{P}^{x}(M_t\in B)\pi(dx)=\pi(B),\quad t\geq0,\  B\in\mathcal{S}.$$  A set $B\in\mathcal{F}$ is said to be
\emph{shift-invariant} if $\theta_t^{-1}B=B$ for all $t\geq0$. The
\emph{shift-invariant} $\sigma$-algebra $\mathcal{I}$ is a
collection of all such shift-invariant sets. The process
$\process{M}$ is said to be \emph{ergodic} if it possesses an invariant probability measure $\pi(dx)$ and if  $\mathcal{I}$ is
trivial with respect to $\mathbb{P}^{\pi}(d\omega)$, that is,
$\mathbb{P}^{\pi}(B)=0$ or $1$ for every $B\in\mathcal{I}$. Here,
for a probability measure $\mu(dx)$ on $\mathcal{S}$,
$\mathbb{P}^{\mu}(d\omega)$ is defined as
$$\mathbb{P}^{\mu}(d\omega):=\int_S\mathbb{P}^{x}(d\omega)\mu(dx).$$
Equivalently, $\process{M}$ is ergodic if it  possesses an invariant probability measure $\pi(dx)$ and if all  bounded harmonic functions are constant $\pi$-a.s.  (see \cite{meyn-tweedie-2009}). Recall,
a bounded measurable function $f(x)$ is called \emph{harmonic} (with respect to $\process{M}$) if $$\int_S\mathbb{P}^{x}(M_t\in dy)f(y)=f(x),\quad x\in S,\ t\geq0.$$
The process
$\process{M}$ is said to be \emph{strongly ergodic} if it possesses an invariant probability measure $\pi(dx)$ and if
$$\lim_{t\longrightarrow\infty}||\mathbb{P}^{x}(M_t\in\cdot)-\pi(\cdot)||_{TV}=0,\quad  x\in S,$$ where $||\cdot||_{TV}$ denotes the
total variation norm on the space of  signed measures on $\mathcal{S}$. Recall that strong ergodicity implies ergodicity (see \cite[Proposition 2.5]{bhat}). On the other hand, ergodicity does not necessarily imply strong ergodicity (for example, see Remark \ref{rm3}).

\section{Main Results}
\label{sec-MainResults}

\ \ \ \ Before stating the main results of this paper, we  introduce some notation we need.
Let $\tau:=(\tau_1,\ldots, \tau_{\bar{d}})\in (0,\infty)^{\bar{d}}$ be
fixed and let $\tau\ZZ^{\bar{d}}:=\tau_1\ZZ\times\ldots\times\tau_{\bar{d}} \ZZ.$
For $x\in\R^{\bar{d}}$, we define
$$x_\tau:=\{y\in\R^{\bar{d}}:x-y\in\tau\ZZ^{\bar{d}}\}\quad\textrm{and}\quad
\R^{\bar{d}}/\tau\ZZ^{\bar{d}}:=\{x_\tau:x\in\R^{\bar{d}}\}.$$ Clearly,
$\R^{\bar{d}}/\tau\ZZ^{\bar{d}}$ is obtained
 by identifying the opposite
faces of $[0,\tau]:=[0,\tau_1]\times\ldots\times[0,\tau_{\bar{d}}]$. Next,
let
 $\Pi_{\tau} : \R^{\bar{d}}\longrightarrow [0,\tau]$, $\Pi_{\tau}(x):=x_\tau$, be the covering map. A
function $f:\R^{\bar{d}}\longrightarrow\R$ is called $\tau$-periodic if
$f\circ\Pi_\tau(x)=f(x)$ for  all $x\in\R^{\bar{d}}$.
For an arbitrary $\tau$-periodic function
$f:\R^{\bar{d}}\longrightarrow\R$, by $f_\tau(x_\tau)$ we denote the
restriction of $f(x)$ to $[0,\tau]$.

We now state the main results of this paper,
 the proofs of which  are given in
Sections \ref{sec-Proofs1} and \ref{sec-Proofs2}.

\begin{theorem}\label{tm1.1}Let  $\process{F}$ be a $\bar{d}$-dimensional diffusion with jumps with  symbol $q(x,\xi)$
  which, in addition, satisfies:
 \begin{description}
  \item [(\textbf{C4})]  the function $x\longmapsto q(x,\xi)$ is $\tau$-periodic for all $\xi\in \R^{\bar{d}}$, or, equivalently, the corresponding L\'evy triplet $(b(x),c(x),\nu(x,dy))$ is $\tau$-periodic;
   \item [(\textbf{C5})]$\process{F}$ possesses a transition
density function $p(t,x,y)$, that is,
$$P_tf(x)=\int_{\R^{\bar{d}}}f(y)p(t,x,y)dy,\quad t>0,\ x\in\R^{\bar{d}},\ f\in B_b(\R^{\bar{d}}),$$ such that $(x,y)\longmapsto p(t,x,y)$ is continuous and $p(t,x,y)>0$
for all $t>0$ and all $x,y\in \R^{\bar{d}}$.
 \end{description}
Then, for any probability measure $\varrho(dx)$ on
$\mathcal{B}(\R^{d})$ having finite first moment, any initial
distribution $\rho(dx)$ of $\process{F}$ and any $\tau$-periodic
$w^{1},\ldots,w^{d}\in C_b^{2}(\R^{\bar{d}})$,
\begin{align}\label{eq1.10}\frac{X_{nt}}{n}\stackrel{\hbox{\scriptsize{$\mathbb{P}_V^{\rho}\times\mathbb{P}^{\varrho}_B\, \textrm{-}\, \textrm{a.s.}$}}}{\xrightarrow{\hspace*{1.8cm}}}\bar{V} t,\end{align}
as $n\longrightarrow\infty$,  and, under
$\mathbb{P}_V^{\rho}\times\mathbb{P}^{\varrho}_B(d\omega_V,d\omega_B)$,
\begin{align}\label{eq1.11}\left\{n^{\frac{1}{2}}\left(\frac{X_{nt}}{n}-\bar{V} t\right)\right\}_{t\geq0}\stackrel{\hbox{\scriptsize{$\textrm{d}$}}}{\longrightarrow}\process{W}, \end{align}
as $n\longrightarrow\infty$,
where
\begin{align}\label{eqV}
\bar{V}:=\left(\int_{\R^{d}}x_1\varrho(dx),  \ldots,\int_{\R^{d}}x_d\varrho(dx)\right) .
\end{align}
 Here,  $C^{k}_b(\R^{\bar{d}})$,
$k\geq0$, denotes the space of $k$ times differentiable functions
such that all derivatives up to order $k$ are bounded,
$V(t,x,\omega_V)=(\mathcal{A}w^{1}(F_t(\omega_V)),\ldots,\mathcal{A}w^{d}(F_t(\omega_V)))$, where the
operator $\mathcal{A}$ is defined by $(\ref{eq1.5})$, and
$\process{W}$ is a $d$-dimensional zero-drift Brownian motion
determined by a covariance matrix  of the form
\begin{align} \label{eq1.12} C:=\Bigg(&\sigma_{ij}+\int_{[0,\tau]}\Big[\langle \nabla w^{i}(x_\tau),
c(x_\tau)\nabla
w^{j}(x_\tau)\rangle\nonumber\\&+\int_{\R^{\bar{d}}}\left(w^{i}(y+x_\tau)-w^{i}(x_\tau)\right)\left(w^{j}(y+x_\tau)-w^{j}(x_\tau)\right)\nu(x_\tau,dy)\Big]\pi_\tau(dx_\tau)\Bigg)_{1\leq
i,j\leq d},
\end{align} where $\pi_\tau(dx_\tau)$ is an invariant measure associated with the projection of $\process{F}$, with respect to $\Pi_\tau(x)$, on $[0,\tau]$.
\end{theorem}

\begin{remark} \label{rm1}
{\rm Two nontrivial examples of diffusions with jumps
satisfying the conditions in $(\textbf{C5})$ can be found in the classes of  diffusions and
stable-like processes. If $\process{F}$ is a
 diffusion with  L\'evy triplet $(b(x),c(x),0)$, such that
$\inf_{x\in\R^{\bar{d}}}\langle\xi,c(x)\xi\rangle\geq c|\xi|^{2}$,
for some $c>0$ and all $\xi\in\R^{\bar{d}}$, and
 $b(x)$ and $c(x)$ are H\"older continuous with the index
$0<\beta\leq1$, then, in \cite[Theorem A]{sheu}, it has been proven
that $\process{F}$ possesses a continuous (in space variables) and strictly positive   transition density
function. Note that, because of the Feller property,
$\mathcal{A}^{\infty}(C_c^{\infty}(\R^{\bar{d}}))\subseteq
C_\infty(\R^{\bar{d}})$,  hence $b(x)$ and $c(x)$ are always
continuous functions.

Let $\alpha:\R^{\bar{d}}\longrightarrow(0,2)$ and
$\gamma:\R^{\bar{d}}\longrightarrow(0,\infty)$ be arbitrarily bounded and
continuously differentiable functions with bounded derivatives, such
that $$0<\inf_{x\in\R^{\bar{d}}}\alpha(x)\leq\sup_{x\in\R^{\bar{d}}}\alpha(x)<2\quad
\textrm{and}\quad \inf_{x\in\R^{\bar{d}}}\gamma(x)>0.$$ Under these
assumptions, in \cite{bass-stablelike}, \cite[Theorem
                                                         5.1]{vasili-stablelike} and
                                                         \cite[Theorem
                                                         3.3.]{rene-wang-feller}
                                                         it has been shown
                                                         that there
                                                         exists a
                                                         unique
                                                         diffusion with jumps,
                                                         called a
                                                         \emph{stable-like
                                                         process},
                                                         determined
by a symbol of the form $$q(x,\xi)=\gamma(x)|\xi|^{\alpha(x)}$$
which satisfies condition  $(\textbf{C5})$. Note that when $\alpha(x)$ and $\gamma(x)$ are
constant functions,  we deal with a rotationally invariant
stable L\'evy process.}
\hfill $\Box$
\end{remark}

\begin{remark} \label{rm2}
{\rm In $(\textbf{C5})$ we assume the existence, continuity $($in space variables$)$ and strict positivity of a transition density function $p(t,x,y)$ of $\process{F}$. According to \cite[Theorem 2.6]{sandric-tams},  the existence  of $p(t,x,y)$ also follows from
\begin{align}\label{eq3.01}\int_{\R^{\bar{d}}}\exp\left[-t\, \inf_{x\in\R^{\bar{d}}}\rm{Re}\,\it{q}(x,\xi)\right]d\xi<\infty,\quad t>0,\ x\in\R^{\bar{d}},\end{align}
 under
\begin{align}\label{eq3.02}\sup_{x\in\R^{\bar{d}}}|{\rm Im}\, q(x,\xi)|\leq c\inf_{x\in\R^{\bar{d}}}{\rm Re}\, q(x,\xi)\end{align} for some $0\leq c<1$ and all $\xi\in\R^{\bar{d}}.$
 According to \cite{friedman} and \cite[Theorems 7.10 and
8.1]{sato-book}, in the L\'evy process and diffusion cases, in order to ensure the
existence of a transition density function,  (\ref{eq3.02}) is not necessary.
Further, note that $\eqref{eq3.01}$ and $(\ref{eq3.02})$ also imply the continuity of $(x,y)\longmapsto p(t,x,y)$ for all $t>0$. Indeed,
 according to
\cite[Theorem 2.7]{rene-wang-feller}, we have
\begin{align}\label{eq2.4}\sup_{x\in\R^{\bar{d}}}\left|\mathbb{E}^{x}\left[e^{i\xi(F_t-x)}\right]\right|\leq\exp\left[-\frac{t}{16}\inf_{x\in\R^{\bar{d}}}{\rm Re}\,\it{q}(x,\rm{2}\xi)\right],\quad
t>0,\ \xi\in\R^{\bar{d}}.\end{align} Thus, from
 \eqref{eq3.01} and \cite[Proposition
2.5]{sato-book}, we have
\begin{align}\label{eq2.5}p(t,x,y)=(2\pi)^{-\bar{d}}\int_{\R^{\bar{d}}}e^{-i\xi(y-x)}\mathbb{E}^{x}\left[e^{i\xi(F_t-x)}\right]d\xi,\quad
t>0,\ x,y\in\R^{\bar{d}}.\end{align} Next, by \cite[Theorem
3.2]{rene-conserv}, the  function
$x\longmapsto\mathbb{E}^{x}\left[e^{i\xi(F_t-x)}\right]$ is
continuous for all $\xi\in\R^{\bar{d}}$. Finally, the continuity of
$(x,y)\longmapsto p(t,x,y)$, $t>0$, follows directly from
\eqref{eq3.01}, $(\ref{eq2.4})$, $(\ref{eq2.5})$ and the dominated convergence
theorem.
On the other hand, the strict positivity of the transition density function
$p(t,x,y)$ is a more complex problem. In the  L\'evy process and
diffusion case this problem has been considered in \cite{brockett},
\cite{rajput}, \cite{friedman}, \cite{sharpe}  and \cite{sheu}. In
the general case, the best we were able to prove is given in Proposition \ref{p2.2} in Section \ref{sec-Discussions}.}
\hfill $\Box$
\end{remark}

In
Theorem \ref{tm1.1} the strong ergodicity is hidden in  assumptions
 (\textbf{C4}) and (\textbf{C5}) (see  Section \ref{sec-Proofs1}).
In Theorem \ref{tm1.2}, we assume (strong) ergodicity directly and show the LLN and CLT hold.
From the physical point of view, ergodicity  is a
natural property of turbulent  flows. Namely,    a system is ergodic
if the underlying process visits  every region of the state space.
On the other hand, very turbulent  flows (with high Reynolds
numbers) are characterized by a low  momentum diffusion and high
momentum advection. In other words, a fluid particle in a very
turbulent fluid has a tendency to visit all regions of the state
space.

\begin{theorem}\label{tm1.2}Let  $\process{F}$ be a $\bar{d}$-dimensional  diffusion with jumps and let  $w^{1},\ldots,w^{d}\in C_b^{2}(\R^{\bar{d}})$ be arbitrary. If $\process{F}$ is ergodic with an invariant probability measure $\pi(dx)$,
then  there exists a $\pi(dx)$ measure zero
set $B\in\mathcal{B}(\R^{\bar{d}})$, such that for any probability
measure $\varrho(dx)$ on $\mathcal{B}(\R^{d})$ having finite first
moment and any initial distributions $\rho(dx)$  of
$\process{F}$, satisfying $\rho(B)=0$, we have
\begin{align}\label{eq1.13}\frac{X_{nt}}{n}\stackrel{\hbox{\scriptsize{$\mathbb{P}_V^{\rho}\times\mathbb{P}^{\varrho}_B\,
\textrm{-}\,
\textrm{a.s.}$}}}{\xrightarrow{\hspace*{1.8cm}}}\bar{V} t , \end{align}
as $n\longrightarrow\infty$,  and, under
$\mathbb{P}_V^{\rho}\times\mathbb{P}^{\varrho}_B(d\omega_V,d\omega_B)$,
\begin{align}\label{eq1.14}\left\{n^{\frac{1}{2}}\left(\frac{X_{nt}}{n}-\bar{V}
t\right)\right\}_{t\geq0}\stackrel{\hbox{\scriptsize{$\textrm{d}$}}}{\longrightarrow}\process{W} ,\end{align}
as $n\longrightarrow\infty$. Here, $\bar{V}$ is given in \eqref{eqV},
$V(t,x,\omega_V)=(\mathcal{A}w^{1}(F_t(\omega_V)),\ldots,\mathcal{A}w^{d}(F_t(\omega_V)))$, where the
operator $\mathcal{A}$ is defined by \eqref{eq1.5}, and
$\process{W}$ is a $d$-dimensional zero-drift Brownian motion
determined by a covariance matrix  of the form
\eqref{eq1.12}, with $\pi(dx)$ instead of $\pi_\tau(dx_\tau)$. In
addition, if $\process{F}$ is strongly ergodic, then the above
convergences hold  for any initial distribution of
$\process{F}.$
\end{theorem}

Note that  diffusions satisfy the assumptions in  (\textbf{C5})
 (see  \cite{rogersI} and
\cite[Theorem A]{sheu}). Hence, Theorems \ref{tm1.1} and \ref{tm1.2}
generalize the results related to diffusions, presented in
  \cite{bhat} and \cite{pap-str-var}. Also, note that  Theorem \ref{tm1.2} is
not applicable to  L\'evy processes, since a (non-trivial) L\'evy
process is never ergodic. On the other hand, in the  L\'evy
process case,  we can relax the
assumptions in (\textbf{C5}), that is, in order to derive the limiting behaviors in (\ref{eq1.3}) and
(\ref{eq1.4}) the strong ergodicity will not be crucial
anymore.   Because of space homogeneity of L\'evy processes,
the assumption in (\textbf{C4}) is automatically
satisfied. First, recall that for a $\tau$-periodic locally
integrable function $f(x)$ its Fourier coefficients are defined by
$$\hat{f}(k):=\frac{1}{|\tau|}\int_{[0,\tau]}e^{-i\frac{2\pi \langle
k, x\rangle}{|\tau|}}f(x)dx,\quad k\in\ZZ^{\bar{d}},$$ where
$|\tau|:=\tau_1\tau_2\cdots\tau_{\bar{d}}.$ Under the assumption
that
\begin{align}\label{eq1.16}
\sum_{k\in\ZZ^{\bar{d}}}|\hat{f}(k)|<\infty,\end{align}
$f(x)=\sum_{k\in\ZZ^{\bar{d}}}\hat{f}(k)e^{i\frac{2\pi \langle k,
x\rangle}{|\tau|}}.$ For example, (\ref{eq1.16}) is satisfied if
$f\in C_b^{1}(\R^{\bar{d}})$ (see \cite[Theorems 3.2.9 and
3.2.16]{grafakos}). In general, $
\sum_{k\in\ZZ^{\bar{d}}}|\hat{f}(k)||k|^{n}<\infty$, $n\geq0$, if
$f\in C_b^{n+1}(\R^{\bar{d}})$.  Recall that we use notation $\process{L}$ instead of $\process{F}$ for L\'evy processes as the driving process of $\process{X}$.

\begin{theorem}\label{tm1.3}Let  $\process{L}$ be a $\bar{d}$-dimensional  L\'evy
process with symbol $q(\xi)$ satisfying \begin{align}\label{eq1.17}\rm{Re}\,\it
{q}\left(\frac{\rm{2}\pi
\it{k}}{|\tau|}\right)>\rm{0},\quad \it{k}\in\ZZ^{\bar{d}}\setminus\{\rm{0}\},
\end{align}
  and let $w^{1},\ldots,w^{d}\in
C_b^{2}(\R^{\bar{d}})$ be $\tau$-periodic.
Then, there exists a Lebesgue measure zero
set $B\in\mathcal{B}(\R^{\bar{d}})$, such that for any probability
measure $\varrho(dx)$ on $\mathcal{B}(\R^{d})$ having finite first
moment and any initial distribution $\rho(dx)$  of
$\process{L}$, satisfying $\rho(B)=0$, we have
\begin{align}\label{eq1.18}\frac{X_{nt}}{n}\stackrel{\hbox{\scriptsize{$\mathbb{P}_V^{\rho}\times\mathbb{P}^{\varrho}_B\,
\textrm{-}\,
\textrm{a.s.}$}}}{\xrightarrow{\hspace*{1.8cm}}}\bar{V} t ,
\end{align}
as $n\longrightarrow\infty$,  and, under
$\mathbb{P}_V^{\rho}\times\mathbb{P}^{\varrho}_B(d\omega_V,d\omega_B)$,
\begin{align}\label{eq1.19}\left\{n^{\frac{1}{2}}\left(\frac{X_{nt}}{n}-\bar{V}
t\right)\right\}_{t\geq0}\stackrel{\hbox{\scriptsize{$\textrm{d}$}}}{\longrightarrow}\process{W} , \end{align}
as $n\longrightarrow\infty$. Here, $\bar{V}$ is given in \eqref{eqV},
$V(t,x,\omega_V)=(\mathcal{A}w^{1}(L_t(\omega_V)),\ldots,\mathcal{A}w^{d}(L_t(\omega_V)))$, where the
operator $\mathcal{A}$ is defined by \eqref{eq1.5}, and
$\process{W}$ is a $d$-dimensional zero-drift Brownian motion
determined by the covariance matrix given in \eqref{eq1.12} with $\pi_\tau(dx_\tau)=dx_\tau/|\tau|$.
In addition, if
\begin{align}\label{eq1.20}\sum_{k\in\ZZ^{\bar{d}}\setminus\{0\}}|k|^{2}|\hat{w}^{i}(k)|\left(1+\left(\rm{Re}\,\it
{q}\left(\frac{\rm{2}\pi
\it{k}}{|\tau|}\right)\right)^{-2}\right)<\infty,\quad
i=1,\ldots,d,\end{align} then the convergence in \eqref{eq1.19} holds for any  initial distribution   of
$\process{L}$.
\end{theorem}

\begin{remark} \label{rm}
{\rm Note that when $${\rm Re}\,
q\left(\frac{2\pi
k}{|\tau|}\right)>0,\ \it{k}\in\ZZ^{\bar{d}}\setminus\{\rm{0}\}, \quad\textrm{and}\quad \liminf_{|{\it k}|\longrightarrow\infty}{\rm Re}\,
{\it q}\left(\frac{2\pi
{\it k}}{|\tau|}\right)>0,$$
the condition in \eqref{eq1.20} reduces to $$\sum_{k\in\ZZ^{\bar{d}}\setminus\{0\}}|k|^{2}|\hat{w}^{i}(k)<\infty,\quad
i=1,\ldots,d.$$ For example, this is the case when the function $ \xi\longmapsto\rm{Re}\,\it
{q}(\xi)$ is radial and the function $ |\xi|\longmapsto\rm{Re}\,\it
{q}(\xi)$ is nondecreasing.
}
\end{remark}

\begin{remark} \label{rm3}
{\rm A simple example  where Theorem \ref{tm1.1} is not applicable, while Theorem
\ref{tm1.3} gives an answer is as follows. Let $w(x)=\sin x$ and let
$\process{L}$ be a one-dimensional L\'evy process given by
L\'evy triplet of the form $(0,0,\delta_{-1}(dy)+\delta_1(dy))$.
Then, clearly, $$\hat{w}(k)=\left\{
                              \begin{array}{ll}
                                \frac{1}{2}, & k=-1,1, \\
                                0, & \textrm{otherwise},
                              \end{array}
                            \right.\quad\textrm{and}\quad q(\xi)=2(1-\cos\xi).$$
Further, note that $q(k)\neq0$ for all $k\in\ZZ\setminus\{0\}.$
Thus, the condition in $(\ref{eq1.20})$ $($and $(\ref{eq1.17})$$)$ holds true and consequently for any initial distribution of $\process{L}$,
$$\left\{n^{-\frac{1}{2}}\int_0^{nt}\mathcal{A}^{b}w(L_s)ds\right\}_{t\geq0}\stackrel{\hbox{\scriptsize{$\textrm{d}$}}}{\longrightarrow}\{W_t\}_{t\geq0},$$
where $\process{W}$ is a zero-drift Brownian motion
with the variance parameter $C=2(1-\cos1).$ Also, note that, according to  Proposition \ref{p2.7} below, $\process{L^{2\pi}}$ is ergodic  but obviously it is not strongly ergodic $($with respect to $dx_{2\pi}/2\pi$$)$.} \hfill $\Box$
\end{remark}

\begin{remark} \label{rm4}
{\rm In Theorem \ref{tm1.1} we implicitly assume (through  conditions (\textbf{C4}) and (\textbf{C5})) that the underlying process $\process{F^{\tau}}$ is strongly ergodic and conclude the limiting behaviors in $(\ref{eq1.10})$ and $(\ref{eq1.11})$ for any initial distribution of $\process{F}$. In Theorem \ref{tm1.3}  we implicitly assume $($through $(\ref{eq1.17})$$)$ only the ergodicity of $\process{L^{\tau}}$ and   the best we can conclude is that the limiting behaviors in $(\ref{eq1.18})$ and $(\ref{eq1.19})$ hold for any initial distribution of $\process{L}$ whose overall mass is contained in the complement of a certain Lebesgue measure zero set.
(See more discussions on the condition \ref{eq1.17} in Section \ref{sec-Comment}.)
If, in addition, we assume that $\process{L}$  satisfies $(\textbf{C5})$, then $\process{L^{\tau}}$ becomes strongly ergodic. Conditions  that certainly  ensure this are the integrability of $e^{-tq(\xi)}$,   $t>0$, and that either the Brownian  or jumping component is nondegenerate and possesses a strictly positive transition density function $($see \cite[Theorem 19.2 and Lemma 27.1]{sato-book}$)$. Note that for the jumping part to possess a transition density function it is necessary that $\nu(\R^{\bar{d}})=\infty.$
Very recently in \cite[Theorem 4.1]{rene-jian-coupling} it has been shown that  $\process{L^{\tau}}$ is strongly ergodic  if there exists some $t_0>0$ such that for every $t\geq t_0$, the transition function $p(t,x,dy)$ of $\process{L}$ has $($with respect to the Lebesgue measure$)$ an absolutely continuous component. According to \cite[Theorem 4.3]{rene-jian-coupling}, a sufficient condition that guarantees the existence of an absolutely continuous component of $p(t,x,dy)$, $t>0$, is that there exists $\varepsilon>0$, such that for
$$\nu_\varepsilon(B):=\left\{
                                                                                                                        \begin{array}{ll}
                                                                                                                          \nu(B), & \nu(\R^{\bar{d}})<\infty , \\
                                                                                                                          \nu(\{x\in B:|x|\geq\varepsilon\}), & \nu(\R^{\bar{d}})=\infty ,
                                                                                                                        \end{array}
                                                                                                                      \right.$$
either the $k$-fold convolution $\nu_\varepsilon^{\ast k}(dy)$, $k\geq1$, has an absolutely continuous component or there exist $\eta>0$ and $k\geq1$, such that \begin{align}\label{eq2.23}\inf_{x\in\R^{\bar{d}},\, |x|\leq\eta}\nu_\varepsilon^{\ast k}\wedge(\delta_x\ast\nu_\varepsilon^{\ast k})(\R^{\bar{d}})>0.\end{align} Here, for two probability measures $\rho(dx)$ and $\varrho(dx)$, $(\rho\wedge\varrho)(dx):=\rho(dx)-(\rho-\varrho)^{+}(dx),$ where  $(\rho-\varrho)^{\pm}(dx)$ is the Hahn-Jordan decomposition of the signed measure $(\rho-\varrho)(dx).$
Intuitively, condition $(\ref{eq2.23})$ ensures enough jump activity of  the underlying pure jump L\'evy process.}
\hfill $\Box$
\end{remark}


\section{Proofs of Theorems \ref{tm1.1} and \ref{tm1.2}} \label{sec-Proofs1}

\subsection{Preliminaries on Periodic Diffusions  with Jumps} \label{sec-Periodic}

\ \ \ \ We start this subsection with the following observation.
 Let $\process{M}$ be an $\R^{\bar{d}}$-valued, $\bar{d}\geq1$, Markov process with semigroup $\process{P}$ and let  $\Pi_{\tau}:\R^{\bar{d}}\longrightarrow[0,\tau]$ be the covering map, defined in the previous section. Recall that $\tau:=(\tau_1,\ldots,\tau_{\bar{d}})\in(0,\infty)^{\bar{d}}$ and $[0,\tau]:=[0,\tau_1]\times\cdots[0,\tau_{\bar{d}}].$
Next, denote by $\process{M^{\tau}}$  the process on $[0,\tau]$
obtained by the projection of the process $\process{M}$ with respect
to $\Pi_{\tau}(x)$, that is, $M_t^{\tau}:=\Pi_{\tau}(M_t),$
$t\geq0$. Then, if $\process{M}$ is ``$\tau$-periodic",
  $\process{M^{\tau}}$ is a Markov process.
More precisely,
  by assuming that
\begin{description}
 \item [(\textbf{A1})] $\process{P}$ preserves the class of all
  $\tau$-periodic  functions in $B_b(\R^{\bar{d}})$, that is, $x\longmapsto P_tf(x)$ is
  $\tau$-periodic for  all $t\geq0$ and all $\tau$-periodic $f\in B_b(\R^{\bar{d}})$,
\end{description}
 by \cite[Proposition 3.8.3]{vasili-book}, the
process $\process{M^{\tau}}$ is a
 Markov process on $([0,\tau],\mathcal{B}([0,\tau])$ with positivity preserving contraction semigroup $\process{P^{\tau}}$
on the space $(B_b([0,\tau]),||\cdot||_\infty)$ given by
$$P_t^{\tau}f_\tau(x_\tau):=\mathbb{E}^{x_\tau}_\tau[f_\tau(M_t^{\tau})]=\int_{[0,\tau]}f_\tau(y_\tau)\mathbb{P}^{x_\tau}_{\tau}(M^{\tau}_t\in dy_\tau),\quad t\geq0,\ x_\tau\in[0,\tau],\ f_\tau\in B_b([0,\tau]),$$ where
\begin{align}\label{eq2.1}\mathbb{P}^{x_\tau}_{\tau}(M^{\tau}_t\in dy_\tau):=\sum_{k\in\tau\ZZ^{\bar{d}}}\mathbb{P}^{x}(M_t\in dy+k),\quad t>0,\ x_\tau,y_\tau\in[0,\tau],\end{align}  and    $x$ and $y$ are arbitrary points  in
$\Pi^{-1}_{\tau}(\{x_\tau\})$ and $\Pi^{-1}_{\tau}(\{y_\tau\})$,
respectively. Note that  $\mathcal{B}([0,\tau])$ can be identified
with the sub $\sigma$-algebra of ``$\tau$-periodic" sets in
$\mathcal{B}(\R^{\bar{d}})$ (that is, the sets whose characteristic
function is $\tau$-periodic) by the relation
$$B=\bigcup_{k\in\tau\ZZ^{\bar{d}}}B_\tau+k,$$ where
$B_\tau\in\mathcal{B}([0,\tau])$ and $B\in\mathcal{B}(\R^{\bar{d}})$
is ``$\tau$-periodic".
 Further, since $[0,\tau]$ is compact, it is reasonable to expect
that $\process{M^{\tau}}$ is (strongly) ergodic. By assuming, in
addition, that
\begin{description}
 \item [(\textbf{A2})]  $\process{M}$ possesses a transition
density function $p(t,x,y)$, that is,
$$P_tf(x)=\int_{\R^{\bar{d}}}f(y)p(t,x,y)dy,\quad t>0,\ x\in\R^{\bar{d}},\ f\in B_b(\R^{\bar{d}}),$$
\item[(\textbf{A3})]
  $(x,y)\longmapsto p(t,x,y)$ is continuous and $p(t,x,y)>0$
for all $t>0$ and all $x,y\in \R^{\bar{d}}$,
\end{description}
then, clearly,  $\process{M^{\tau}}$ has a transition density function $ p_{\tau}(t,x_{\tau},y_\tau)$, which is, according to (\ref{eq2.1}), given by
$$p_{\tau}(t,x_{\tau},y_\tau)=\sum_{k\in\tau\ZZ^{\bar{d}}}p(t,x_\tau,y_\tau+k),\quad \ t>0,\ x_\tau,y_\tau\in[0,\tau],$$ and, by
 (\textbf{A3}),  it satisfies
$$\inf_{x_\tau,y_\tau\in[0,\tau]}p_{\tau}(t,x_\tau,y_\tau)>0,\quad t>0.$$
Thus,
 by \cite[the proof of Theorem
III.3.1]{benso-lions-book}, the process $\process{M^{\tau}}$
possesses a unique invariant probability measure $\pi_\tau(dx_\tau)$, such
that
\begin{align}\label{eq2.2}\sup\left\{\left|P^{\tau}_t1_{B_\tau}(x_\tau)-\pi_\tau(B_\tau)\right|:x_\tau\in[0,\tau],\
B_\tau\in\mathcal{B}([0,\tau])\right\}\leq \Lambda e^{-\lambda t}\end{align} for
all $t\geq0$ and some universal constants $\lambda>0$ and $\Lambda>0$. In
particular, $\process{M^{\tau}}$ is strongly ergodic. Let us remark
that, under (\textbf{A2}), (\textbf{A1}) holds true if the function
$x\longmapsto p(t,x,y+x)$ is $\tau$-periodic for all $t>0$ and all
$y\in\R^{\bar{d}}$.

Now, based on the above observations, we prove that a  diffusion with jumps which satisfies (\textbf{C4}) and (\textbf{C5}) also satisfies the conditions in (\textbf{A1}), (\textbf{A2}) and (\textbf{A3}).
\begin{proposition}\label{p2.1} Let $\process{F}$ be a diffusion with jumps with  L\'evy triplet $(b(x),c(x),\nu(x,dy))$ and transition density function $p(t,x,y)$. Then, $\process{F}$ satisfies the condition in $(\textbf{C4})$ if, and only if, the function $x\longmapsto p(t,x,x+y)$ is $\tau$-periodic for all $t>0$ and all $y\in\R^{\bar{d}}.$ \end{proposition}
\begin{proof}
The sufficiency follows directly from  \cite[the proof of Theorem 4.5.21]{jacobI}. To prove the necessity, first recall that  there exists a suitable enlargement of the stochastic basis $(\Omega,\mathcal{F},\{\mathbb{P}^{x}\}_{x\in\R^{\bar{d}}},$ $\process{\mathcal{F}},\process{\theta})$, say $(\widetilde{\Omega},\mathcal{\widetilde{F}},\{\mathbb{\widetilde{P}}^{x}\}_{x\in\R^{\bar{d}}},\process{\mathcal{\widetilde{F}}},\process{\widetilde{\theta}})$, on which $\process{F}$ is
the solution to the following stochastic differential equation \newpage
\begin{align}\label{eq4.3}F_t=&x+\int_0^{t}b(F_{s-})ds+\int_0^{t}c(F_{s-})d\widetilde{W}_s\nonumber\\&+\int_0^{t}\int_{\R\setminus\{0\}}k(F_{s-},z)1_{\{u:|k(F_{s-},u)|\leq1\}}(z)\left(\widetilde{\mu}(\cdot,ds,dz)-ds\widetilde{N}(dz)\right)\nonumber\\&+\int_0^{t}\int_{\R\setminus\{0\}}k(F_{s-},z)1_{\{u:|k(F_{s-},u)|>1\}}(z)\,\widetilde{\mu}(\cdot,ds,dz),
\end{align}
where $\process{\widetilde{W}}$ is a $\bar{d}$-dimensional Brownian motion, $\widetilde{\mu}(\omega,ds,dz)$
is a Poisson random measure with compensator (dual
predictable projection) $ds\widetilde{N}(dz)$ and $k:\R^{\bar{d}}\times\R\setminus\{0\}\longrightarrow\R^{\bar{d}}$ is a Borel measurable function  satisfying
$$\widetilde{\mu}(\omega,ds,k(F_{s-}(\omega),\cdot)\in dy)=\sum_{s: \Delta F_s(\omega)\neq0}\delta_{(s,\Delta F_s(\omega))}(ds,dy),$$
$$ds\widetilde{N}(k(F_{s-}(\omega),\cdot)\in dy)=ds\,\nu(F_{s-}(\omega),dy)$$ (see \cite[Theorem 3.5]{rene-holder} and \cite[Theorem 3.33]{cinlar}).
Further,  $\process{F}$ has the same transition function on the starting and enlarged stochastic basis. Thus, because of the $\tau$-periodicity of $(b(x),c(x),\nu(x,dy))$, directly from (\ref{eq4.3}) we read that $\mathbb{P}^{x+\tau}(F_t\in dy)=\mathbb{P}^{x}(F_t+\tau\in dy)$ for all $t\geq0$ and all $x\in\R^{\bar{d}}$, which proves the assertion.
\end{proof}

Since we mainly deal with $\tau$-periodic functions,  we need to extend the operator $(\mathcal{A}^{\infty}|_{C_c^{\infty}(\R^{\bar{d}})},$ $C_c^{\infty}(\R^{\bar{d}}))$ on a larger domain which contains a certain class of $\tau$-periodic functions.
Recall  that  every Feller semigroup $\process{P}$ can be
uniquely extended to $B_b(\R^{\bar{d}})$.  We denote this
extension again by $\process{P}$.
\begin{proposition}\label{p2.3} Let $\process{F}$ be a diffusion with jumps with $B_b$-generator $(\mathcal{A}^{b},\mathcal{D}_{\mathcal{A}^{b}})$ which satisfies the
condition in $(\textbf{C4})$. Then,
$$\{f\in C_b^{2}(\R^{\bar{d}}):f(x)\ \textrm{is}\
\tau\textrm{-periodic}\}\subseteq\mathcal{D}_{\mathcal{A}^{b}}$$ and,
on this class of functions,
$\mathcal{A}^{b}$ has the representation in \eqref{eq1.5}.
\end{proposition}
\begin{proof}
Let $\mathcal{L}:C_b^{2}(\R^{\bar{d}})\longrightarrow B_b(\R^{\bar{d}})$ be defined by
the relation  in (\ref{eq1.5}). Observe that
actually  $\mathcal{L}:C_b^{2}(\R^{\bar{d}})\longrightarrow C_b(\R^{\bar{d}})$ (see
\cite[Remark 4.5]{rene-conserv}).   Next, by \cite[Corollary
3.6]{rene-holder}, we have
$$\mathbb{E}^{x}\left[f(F_t)-\int_0^{t}\mathcal{L}f(F_s)ds\right]=f(x),\quad
x\in\R^{\bar{d}},\ f\in C^{2}_b(\R^{\bar{d}}).$$ Now, let  $f\in
C^{2}_b(\R^{\bar{d}})$ be $\tau$-periodic. Then, since $x\longmapsto
\mathcal{L}f(x)$ is also $\tau$-periodic, we have
\begin{align*}\lim_{t\longrightarrow0}\left|\left|\frac{P_tf-f}{t}-\mathcal{L}f\right|\right|_{\infty}&=\lim_{t\longrightarrow0}\left|\left|\frac{1}{t}\int_0^{t}(P_s\mathcal{L}f-\mathcal{L}f)ds\right|\right|_\infty\\&
\leq\lim_{t\longrightarrow0}\frac{1}{t}\int_0^{t}\sup_{x\in[0,\tau]}|P_s\mathcal{L}f(x)-\mathcal{L}f(x)|ds\\&=0,\end{align*}
where in the final step we  applied \cite[Lemma 4.8.7]{jacobI}.
\end{proof}

In the following proposition we derive a connection between   the $B_b$-generators $(\mathcal{A}^{b},\mathcal{D}_{\mathcal{A}^{b}})$ and $(\mathcal{A}_{\tau}^{b},\mathcal{D}_{\mathcal{A}_\tau^{b}})$ of $\process{F}$ and $\process{F^{\tau}}$, respectively. Recall that for a $\tau$-periodic function $f(x)$, $f_\tau(x_\tau)$ denotes its restriction to $[0,\tau].$
\begin{proposition}\label{p2.4} Let $\process{F}$ be a diffusion with jumps  satisfying the condition in
$(\textbf{C4})$ and let
$\process{F^{\tau}}$ be the projection of $\process{F}$ on
$[0,\tau]$ with respect to $\Pi_\tau(x)$.  Further, let
$(\mathcal{A}^{b},\mathcal{D}_{\mathcal{A}^{b}})$ and
$(\mathcal{A}_{\tau}^{b},\mathcal{D}_{\mathcal{A}_\tau^{b}})$ be the
$B_b$-generators of $\process{F}$ and $\process{F^{\tau}}$,
respectively. Then, we have $$\{f_\tau:f\in C_b^{2}(\R^{\bar{d}})\
\textrm{and}\ f(x)\ \textrm{is}\
\tau\textrm{-periodic}\}\subseteq\mathcal{D}_{\mathcal{A}_\tau^{b}},$$
and, on this set,
$\mathcal{A}_\tau^{b}f_\tau=\left(\mathcal{A}^{b}f\right)_\tau.$
\end{proposition}
\begin{proof}
First,  according to Proposition \ref{p2.3}, $\{f\in C_b^{2}(\R^{\bar{d}}):f(x)\
\textrm{is}\
\tau\textrm{-periodic}\}\subseteq\mathcal{D}_{\mathcal{A}^{b}}$. This and $\tau$-periodicity automatically yield that for
 any $\tau$-periodic $f\in
C^{2}_b(\R^{\bar{d}})$, we have
$$\lim_{t\longrightarrow0}\left|\left|\frac{P^{\tau}_tf_\tau-f_\tau}{t}-\left(\mathcal{A}^{b}f\right)_\tau\right|\right|_{\infty}=\lim_{t\longrightarrow0}\left|\left|\frac{P_tf-f}{t}-\mathcal{A}^{b}f\right|\right|_{\infty}=0,$$
which proves the desired result.
\end{proof}

\subsection{Proof of Theorem \ref{tm1.1}} \label{sec-prooftm1.1}

\ \ \ \ Before proving the main result of this subsection (Theorem \ref{tm1.1}), let us  recall the notion of characteristics of a semimartingale
(see \cite{jacod}). Let
$(\Omega,\mathcal{F},\process{\mathcal{F}},\mathbb{P},\process{S})$, denoted by
$\process{S}$ in the sequel, be a $d$-dimensional semimatingale and
let $h:\R^{d}\longrightarrow\R^{d}$ be a truncation function (that is, a
continuous bounded function such that $h(x)=x$ in a neighborhood of
the origin).
 We  define two processes
$$\check{S}(h)_t:=\sum_{s\leq t}(\Delta S_s-h(\Delta S_s))\quad
\textrm{and} \quad S(h)_t:=S_t-\check{S}(h)_t,$$ where the process
$\process{\Delta S}$ is defined by $\Delta S_t:=S_t-S_{t-}$ and
$\Delta S_0:=S_0$. The process $\process{S(h)}$ is a \emph{special
semimartingale}, that is, it admits a unique decomposition
\begin{align}\label{eq:2.7}S(h)_t=S_0+M(h)_t+B(h)_t,\end{align} where $\process{M(h)}$ is a local
martingale and $\process{B(h)}$ is a predictable process of bounded
variation.
\begin{definition}
 Let $\process{S}$  be a
semimartingale and let $h:\R^{d}\longrightarrow\R^{d}$ be a truncation
function. Furthermore, let $\process{B(h)}$  be the predictable
process
 of bounded variation appearing in \eqref{eq:2.7},  let $N(\omega,ds,dy)$ be the
compensator of the jump measure
$$\mu(\omega,ds,dy):=\sum_{s:\Delta S_s(\omega)\neq 0}\delta_{(s,\Delta S_s(\omega))}(ds,dy)$$ of the process
$\process{S}$ and let $\process{C}=\{(C_t^{ij})_{1\leq i,j\leq d})\}_{t\geq0}$ be the quadratic co-variation
process for $\process{S^{c}}$ $($continuous martingale part of
$\process{S}$$)$, that is,
$$C^{ij}_t=\langle S^{i,c}_t,S^{j,c}_t\rangle.$$  Then $(B,C,N)$ is called
the \emph{characteristics} of the semimartingale $\process{S}$
 $($relative to $h(x)$$)$. If we put $\tilde{C}(h)^{ij}_t:=\langle
 M(h)^{i}_t,M(h)^{j}_t\rangle$, $i,j=1,\ldots,d$, where $\process{M(h)}$ is the local martingale
 appearing in \eqref{eq:2.7}, then $(B,\tilde{C},N)$ is called the \emph{modified
 characteristics} of the semimartingale $\process{S}$ $($relative to $h(x)$$)$.
\end{definition}

Now, we prove Theorem \ref{tm1.1}.

\begin{proof}[Proof of Theorem \ref{tm1.1}] The proof proceeds  in three steps.

\textbf{Step 1.} In the first step, we explain  our strategy of the
proof. First, note that, because of  the independence of
$\process{F}$ and $\process{B}$, \cite[Theorem 36.5]{sato-book} and Proposition \ref{p2.3}, in order to prove the relation in  (\ref{eq1.10}), it suffices to prove that
\begin{align}\label{eq4}
n^{-1}\int_0^{nt}\mathcal{A}^{b}w^{i}(F_s)ds\stackrel{\hbox{\scriptsize{$\mathbb{P}^{\rho}\,
\textrm{-}\, \textrm{a.s.}$}}}{\xrightarrow{\hspace*{1cm}}} 0
\end{align}
 for
all $t\geq0$, all $i=1,\ldots,d$ and all initial distributions
$\rho(dx)$ of $\process{F}$. Recall that
$w^{1},\ldots,w^{d}\in C^{2}_b(\R^{\bar{d}})$ are $\tau$-periodic.
Next, due to the $\tau$-periodicity of the L\'evy triplet of $\process{F}$ (which implies that $\mathcal{A}^{b}f(x)$ is $\tau$-periodic for any $\tau$-periodic $f\in C^{2}_b(\R^{\bar{d}})$) and  by noting that for any $\tau$-periodic  $f:\R^{\bar{d}}\longrightarrow\R$, $f(F_t)=f(F_t^{\tau})$, $t\geq0$,  we observe that we can replace $\process{F}$ by $\process{F^{\tau}}$ in \eqref{eq4}, which is, by \eqref{eq2.2}, strongly ergodic. Hence,  the limiting behavior
in (\ref{eq4})  will simply follow by employing Proposition \ref{p2.4} and the Birkhoff ergodic theorem.

Similarly as above, because of the
independence of $\process{F}$ and $\process{B}$ and the scaling
property of $\process{B}$ (that is,
$\{B_t\}_{t\geq0}\stackrel{\hbox{\scriptsize{$\textrm{d}$}}}{=}\{c^{-1/2}B_{ct}\}_{t\geq0}$
for all $c>0$), we conclude that   in order to prove the limiting behavior
in (\ref{eq1.11}),   it suffices to prove that for any initial
distribution $\rho(dx)$ of $\process{F}$,
\begin{align}\label{2.8}\left\{n^{-\frac{1}{2}}\int_0^{nt}v(F_s)ds\right\}_{t\geq0}\stackrel{\hbox{\scriptsize{$\textrm{d}$}}}{\longrightarrow}\process{\tilde{W}}\end{align}
under $\mathbb{P}^{\rho}(d\omega_V)$, where
$v(x)=(\mathcal{A}^{b}w^{1}(x),\ldots,\mathcal{A}^{b}w^{d}(x))$ and
$\process{\tilde{W}}$ is a zero-drift Brownian motion determined by
a covariance matrix of the form $\tilde{C}:=C-\Sigma$, where
the matrices $C$ and $\Sigma$ are given in (\ref{eq1.12})  and \eqref{eq1.2}, respectively.
Now, according to
\cite[Theorem VIII.2.17]{jacod}, \eqref{2.8} will follow if we prove   the convergence (in probability) of the  modified characteristics of $\left\{n^{-1/2}\int_0^{nt}v(F_s)ds\right\}_{t\geq0}$ to the modified characteristics of $\process{\tilde{W}}$. Accordingly,  we explicitly compute the modified characteristics of $\left\{n^{-1/2}\int_0^{nt}v(F_s)ds\right\}_{t\geq0}$ (in terms of the L\'evy triplet of $\process{F}$) and, again, because of the $\tau$-periodicity of the L\'evy triplet of $\process{F}$, we  switch from $\process{F}$ to $\process{F^{\tau}}$ and apply the Birkhoff ergodic theorem, which concludes the proof of Theorem \ref{tm1.1}.

\textbf{Step 2.} In the second step, we prove the limiting behavior
in (\ref{eq4}). First, observe that, by Proposition \ref{p2.4}, we have
$$\mathcal{A}^{b}w^{i}(F_t)=\mathcal{A}^{b}w^{i}(F^{\tau}_t)=\left(\mathcal{A}^{b}w^{i}\right)_\tau(F^{\tau}_t)=\mathcal{A}_\tau^{b}w^{i}_\tau(F^{\tau}_t),\quad
t\geq0,\ i=1,\ldots,d.$$ Using this fact, (\ref{eq2.2}) and
\cite[Proposition 2.5]{bhat} (which states that the Birkhoff ergodic theorem for strongly ergodic Markov processes holds for any initial distribution) we conclude that for any
initial distribution $\rho(dx)$ of $\process{F}$, we have
\begin{align*}n^{-1}\int_0^{nt}\mathcal{A}^{b}w^{i}(F_s)ds\stackrel{\hbox{\scriptsize{$\mathbb{P}^{\rho}\, \textrm{-}\, \textrm{a.s.}$}}}{\xrightarrow{\hspace*{1cm}}}t\int_{[0,\tau]}\mathcal{A}_\tau^{b}w^{i}_\tau(x_\tau)\pi_\tau(dx_\tau),\quad i=1,\dots,d.\end{align*} Here, $\pi_\tau(dx_\tau)$ denotes the unique invariant probability measure of $\process{F^{\tau}}$. Finally, we have
\begin{align*}\left|\int_{[0,\tau]}\mathcal{A}^{b}_\tau w^{i}_\tau(x_\tau)\pi_\tau(dx_\tau)\right|&=
\lim_{t\longrightarrow0}\left|\int_{[0,\tau]}\mathcal{A}^{b}_{\tau}w^{i}_{\tau}(x_\tau)\pi_\tau(dx_\tau)-\int_{[0,\tau]}\left(\frac{P^{\tau}_tw^{i}_{\tau}-w^{i}_\tau}{t}\right)(x_\tau)\pi_\tau(dx_\tau)\right|\nonumber\\&
\leq\lim_{t\longrightarrow0}\left|\left|\mathcal{A}^{b}_{\tau}w^{i}_{\tau}-\frac{P^{\tau}_tw^{i}_{\tau}-w^{i}_\tau}{t}\right|\right|_\infty\nonumber\\
&=0,\end{align*} where in the first equality we used the stationarity
property of $\pi_\tau(dx_\tau).$

\textbf{Step 3.} In the third step, we prove the limiting behavior
in (\ref{2.8}).
Let $\rho(dx)$ be an arbitrary initial distribution of
$\process{F}$. According to \cite[Proposition 4.1.7]{ethier}, the
processes
$$
S^{i,n}_t:=n^{-\frac{1}{2}}\int_0^{nt}\mathcal{A}w^{i}(F_s)ds-n^{-\frac{1}{2}}w^{i}(F_{nt})+n^{-\frac{1}{2}}w^{i}(F_0),\quad i=1,\ldots,d,
$$
are  $\mathbb{P}^{\rho}$-martingales (with respect to the natural filtration).
Further, let $h:\R^{d}\longrightarrow\R^{d}$ be
an arbitrary truncation function such that $h(x)=x$ for all
$|x|\leq2\max_{i\in\{1,\ldots,d\}}||w^{i}||_\infty.$ Then,  $S^{n}_t = S^{n}(h)_t$ for all $t\geq0$
and all $n\geq1$, that is, $\process{S^{n}}$ is a special
semimartingale with $S^{n}_0=0$ for all $n\geq1$.
In particular,
$B_t^{n}=0$ for all $t\geq0$ and all $n\geq1$. Now, by applying
It\^{o}'s formula to $S^{n}_t$, directly from  \cite[Theorem
II.2.34]{jacod} and \cite[Theorem 3.5]{rene-holder}, one easily obtains that
\begin{align}\label{eq2.9}C^{i,j,n}_t=n^{-1}\int_0^{nt}\langle \nabla w^{i}(F_{s-}),
c(F_{s-})\nabla w^{j}(F_{s-})\rangle ds,\quad
i,j=1,\ldots,d.\end{align} Since
$w^{i}w^{j}\in\mathcal{D}_{\mathcal{A}^{b}}$, $i,j=1,\ldots d$,
again by \cite[Proposition 4.1.7]{ethier}, the processes
$$
\bar{S}^{i,j,n}_t:=n^{-1}\int_0^{nt}\mathcal{A}^{b}(w^{i}w^{j})(F_s)ds-n^{-1}w^{i}(F_{nt})w^{j}(F_{nt})+n^{-1}w^{i}(F_0)w^{j}(F_0),\quad i,j=1,\ldots,d,
$$
are also $\mathbb{P}^{\rho}$-martingales. According to (a
straightforward adaption of) \cite[Problem 2.19]{ethier}, this
yields
\begin{align*}\tilde{C}^{i,j,n}_t&=\langle
S_t^{i,n},S^{j,n}_t\rangle\\&=n^{-1}\int_0^{nt}\left(\mathcal{A}^{b}(w^{i}w^{j})(F_{s-})-w^{i}(F_{s-})\mathcal{A}^{b}w^{j}(F_{s-})-w^{j}(F_{s-})\mathcal{A}^{b}w^{i}(F_{s-})\right)ds\\&=
n^{-1}\int_0^{nt}\int_{\R^{\bar{d}}}\left(w^{i}(y+F_{s-})-w^{i}(F_{s-})\right)\left(w^{j}(y+F_{s-})-w^{j}(F_{s-})\right)\nu(F_{s-},dy)ds\\
&\ \ \ + n^{-1}\int_0^{nt}\langle \nabla w^{i}(F_{s-}),
c(F_{s-})\nabla w^{j}(F_{s-})\rangle ds,\quad
i,j=1,\ldots,d,\end{align*}
 and, by  (\ref{eq2.9}),
\cite[Proposition II.2.17]{jacod} and \cite[Theorem 3.5]{rene-holder},
 $$N^{n}(\omega,ds,B)=\int_{\R^{\bar{d}}}1_{B}\left(n^{-\frac{1}{2}}w(y+F_{s-}(\omega))-n^{-\frac{1}{2}}w(F_{s-}(\omega))\right)\nu\left(F_{s-}(\omega),dy\right)ds,\quad B\in\mathcal{B}(\R^{\bar{d}}).$$ Here, $w(x):=(w^{1}(x),\ldots,w^{d}(x))$.

Now, we show that under $\mathbb{P}^{\rho}(d\omega_V)$, $$\process{S^{n}}\stackrel{\hbox{\scriptsize{$\textrm{d}$}}}{\longrightarrow}\process{\tilde{W}}.$$   To prove this, according to
\cite[Theorem VIII.2.17]{jacod}, it suffices to prove that
\begin{align}\label{eq2.10}\int_0^{nt}\int_{\R^{d}}|g(y)|N^{n}(\omega,ds,dy)\stackrel{\hbox{\scriptsize{$\mathbb{P}^{\rho}\textrm{-a.s.}$}}}{\xrightarrow{\hspace*{1cm}}} 0\quad
\end{align} for all $t\geq0$ and all $g\in C_b(\R^{d})$
vanishing in a neighborhood around the origin,
and\begin{align}\label{eq2.11}\tilde{C}^{n}_t\stackrel{\hbox{\scriptsize{$\mathbb{P}^{\rho}\textrm{-a.s.}$}}}{\xrightarrow{\hspace*{1cm}}}t\tilde{C}\end{align}
for all $t\geq0$. The relation in (\ref{eq2.10}) easily follows from
the fact that the function $w(x)$ is bounded and $g(x)$ vanishes in
a neighborhood around the origin. To prove the relation in
(\ref{eq2.11}), first note that because of $\tau$-periodicity of
all components,
\begin{align*}\tilde{C}^{i,j,n}_t&=n^{-1}\int_0^{nt}\langle \nabla w^{i}(F^{\tau}_{s-}),
c(F^{\tau}_{s-})\nabla w^{j}(F^{\tau}_{s-})\rangle ds \\ &\ \ \
+n^{-1}\int_0^{nt}\int_{\R^{\bar{d}}}\left(w^{i}(y+F^{\tau}_{s-})-w^{i}(F^{\tau}_{s-})\right)\left(w^{j}(y+F^{\tau}_{s-})-w^{j}(F^{\tau}_{s-})\right)\nu(F^{\tau}_{s-},dy)ds\end{align*}
for all $i,j=1,\ldots,d$.  Now, the desired result again follows
by employing \eqref{eq2.2} and \cite[Proposition 2.5]{bhat}. Finally, since
$w(x)$ is bounded,  \cite[Lemma VI.3.31]{jacod} shows the
convergence in (\ref{2.8}), and thus, (\ref{eq1.11}).
\end{proof}
Finally, note that $\pi_\tau(dx_\tau)$
and $dx_\tau$  are mutually absolutely
continuous. Thus, due to  \cite[Proposition 2.4]{bhat},
$C=\Sigma$ (that is, $\tilde{C}=0$)  if, and only if,
$\mathcal{A}^{b}w^{i}(x)=0$ for all $i=1,\ldots,d$.

\subsection{Proof of Theorem \ref{tm1.2}}
\label{sec-Prooftm1.2}

\ \ \ \ We now prove Theorem \ref{tm1.2}. The main ingredients in the proof of   Theorem \ref{tm1.1} were the $\tau$-periodicity of the
driving diffusion with jumps $\process{F}$ and velocity function $v(x)$ and the fact that all $\tau$-periodic $f\in C_b^{2}(\R^{\bar{d}})$ are contained in the domain of the $B_b$-generator $(\mathcal{A}^{b},\mathcal{D}_{\mathcal{A}^{b}})$ of $\process{F}$ and, on this class of functions, $\mathcal{A}^{b}$ has the representation in \eqref{eq1.5} (Proposition \ref{p2.3}). By having these facts, and assuming (\textbf{C5}), we were able to  switch to the strongly ergodic process $\process{F^{\tau}}$ and apply the Birkhoff ergodic theorem. On the other hand, in Theorem \ref{tm1.2} we simply assume the (strong) ergodicity of a driving diffusion with jumps $\process{F}$. Now, one might conclude that the assertion of Theorem \ref{tm1.2}  automatically follows by employing completely the same arguments as in the proof of Theorem \ref{tm1.1}.
   However, note that in this situation it is not clear that $C_b^{2}(\R^{\bar{d}})$ is contained in $\mathcal{D}_{\mathcal{A}^{b}}$  or that $\mathcal{A}^{b}$ can be uniquely extended to $C_b^{2}(\R^{\bar{d}})$.  In order to resolve this problem, according to \cite[Theorem 2.37]{rene-bjorn-jian}, we employ the following facts: (i)  $C_\infty^{2}(\R^{\bar{d}})$ is contained in the domain of the Feller generator $(\mathcal{A}^{\infty},\mathcal{D}_{\mathcal{A}^{\infty}})$ of
$\process{F}$, (ii) for any $f\in C_b^{2}(\R^{\bar{d}})$ there exists a sequence $\{f_n\}_{n\geq1}\subseteq C_\infty^{2}(\R^{\bar{d}})$, such that
$\mathcal{A}^{\infty}f_n$ converges (pointwise) to  $\mathcal{A}f$, where the operator $\mathcal{A}$ is given by \eqref{eq1.5},  and  (iii) for any $f\in C_b^{2}(\R^{\bar{d}})$ and any initial distribution $\rho(dx)$ of $\process{F}$, $\{\int_0^{t}\mathcal{A}f(F_s)ds-f(F_t)+ f(F_0)\}_{t\geq0}$ is a $\mathbb{P}^{\rho}$-martingale.

\begin{proof}[Proof of Theorem \ref{tm1.2}]
Let $(\mathcal{A}^{\infty},\mathcal{D}_{\mathcal{A}^{\infty}})$ be the Feller generator of $\process{F}$. As we  commented above, due to \cite[Theorem 2.37]{rene-bjorn-jian}, $C_\infty^{2}(\R^{\bar{d}})\subseteq\mathcal{D}_{\mathcal{A}^{\infty}}$ and, on this set, $\mathcal{A}^{\infty}$ has again the representation (\ref{eq1.5}). Furthermore, according to the same reference, $(\mathcal{A}^{\infty},\mathcal{D}_{\mathcal{A}^{\infty}})$ has a unique extension to $C_b^{2}(\R^{\bar{d}})$, denoted by $(\mathcal{A},C_b^{2}(\R^{\bar{d}}))$, satisfying $$\mathcal{A}f(x)=\lim_{n\longrightarrow\infty}\mathcal{A}^{\infty}(f\phi_n)(x),\quad x\in\R^{\bar{d}},\ f\in C_b^{2}(\R^{\bar{d}}),$$ for any sequence $\{\phi_n\}_{n\geq1}\subseteq C_c^{\infty}(\R^{\bar{d}})$ with $1_{\{y\in\R^{\bar{d}}:|y|\leq n\}}(x)\leq\phi_n(x)\leq1$ for all $x\in\R^{\bar{d}}$ and all $n\geq1$. Moreover, $\mathcal{A}$ has  the representation (\ref{eq1.5}). Note that $\mathcal{A}f\in B_b(\R^{\bar{d}})$. Now, by the Birkhoff ergodic theorem and  dominated convergence theorem, for any $f\in C_b^{2}(\R^{\bar{d}})$ we have $$\lim_{n\longrightarrow\infty}n^{-1}\int_0^{nt}\mathcal{A}f(F_s)ds=t\int_{\R^{\bar{d}}}\mathcal{A}f(x)\pi(dx)=\lim_{n\longrightarrow\infty}t\int_{\R^{\bar{d}}}\mathcal{A}^{\infty}(f\phi_n)(x)\pi(dx)=0,\quad \mathbb{P}^{\pi}\textrm{-a.s.},$$
where  $\{\phi_n\}_{n\geq1}\subseteq C_c^{\infty}(\R^{\bar{d}})$ is as above and in the final step we used the stationarity property of $\pi(dx)$.
Thus, for any $w^{1},\ldots,w^{d}\in C_b^{2}(\R^{\bar{d}})$,
$$\mathbb{P}^{\pi}\left(\lim_{n\longrightarrow\infty}n^{-1}\int_0^{nt}\mathcal{A}w^{i}(F_s)ds=0\right)=\int_{\R^{\bar{d}}}\mathbb{P}^{x}\left(\lim_{n\longrightarrow\infty}n^{-1}\int_0^{nt}\mathcal{A}w^{i}(F_s)ds=0\right)\pi(dx)=1$$ for all $i=1,\ldots,d.$ Therefore, there exists a $\pi(dx)$ measure zero set
$B\in \mathcal{B}(\R^{\bar{d}})$ such that
$$\mathbb{P}^{x}\left(\lim_{n\longrightarrow\infty}n^{-1}\int_0^{nt}\mathcal{A}w^{i}(F_s)ds=0\right)=1,\quad x\in
B^{c},\ i=1,\ldots, d,$$ which proves the desired result.

To prove the limiting behavior in (\ref{eq1.14}), again by \cite[Theorem 2.37]{rene-bjorn-jian},
for any   $f\in C_b^{2}(\R^{\bar{d}})$ and any initial distribution $\rho(dx)$ of $\process{F}$, the process $$\left\{\int_0^{t}\mathcal{A}f(F_s)ds-f(F_t)+ f(F_0)\right\}_{t\geq0}$$ is a $\mathbb{P}^{\rho}$-martingale with respect to the natural filtration. Thus, by completely the same approach as in the proof of Theorem \ref{tm1.1}, the desired result  follows.
\end{proof}

\section{Proof of Theorem \ref{tm1.3}} \label{sec-Proofs2}

\ \ \ \ In this section, we prove Theorem \ref{tm1.3}. We start with two auxiliary results we need in the sequel.
First, observe
that
$dx_\tau/|\tau|$ is always an invariant probability measure for $\process{L^{\tau}}$. Indeed, let $t\geq0$ and
$B_\tau\in\mathcal{B}([0,\tau])$ be arbitrary. Then, by
(\ref{eq2.1}) and the space homogeneity property of L\'evy
processes, we have
\begin{align*}\int_{[0,\tau]}\mathbb{P}^{x_\tau}_\tau(L_t^{\tau}\in B_\tau)dx_\tau&=\int_{[0,\tau]}\sum_{k\in\tau\ZZ^{\bar{d}}}\int_{B_\tau+k-x_\tau}p(t,0,y)dydx_\tau\\&
=\int_{[0,\tau]}\sum_{k\in\tau\ZZ^{\bar{d}}}\int_{\R^{\bar{d}}}1_{\{B_\tau+k-x_\tau\}}(y)p(t,0,y)dydx_\tau\\&=\int_{\R^{\bar{d}}}p(t,0,y)dy\int_{B_\tau}dx_\tau\\&=\int_{B_\tau}dx_\tau.\end{align*}
In general, $dx_\tau/|\tau|$ is not necessarily the unique invariant probability measure for $\process{L^{\tau}}$. But, if $\process{L}$ is symmetric, that is, $b=0$ and $\nu(dy)$ is a symmetric measure, and possesses a transition density function (not necessary strictly positive), then $dx_\tau/|\tau|$ is unique (see \cite{ying}).
 Having this fact, in the L\'evy process case,   the covariance  matrix $C$ (given by \eqref{eq1.12}) can be computed in an alternative way.
Recall that $\hat{f}(k)$ denotes the $k$-th, $k\in\ZZ^{\bar{d}}$, Fourier coefficient
of a $\tau$-periodic locally integrable function $f(x)$.

\begin{proposition}\label{p2.6} Let $\process{L}$ be a $\bar{d}$-dimensional L\'evy process with symbol $q(\xi)$ and
$B_b$-generator $(\mathcal{A}^{b},\mathcal{D}_{\mathcal{A}^{b}})$.
Further,  let $w^{1},\ldots,w^{d}\in C_b^{2}(\R^{\bar{d}})$ be
$\tau$-periodic, such that
\begin{align} \label{p2.6a1}
\sum_{k\in\ZZ^{\bar{d}}}|\hat{w}^{i}(k)||k|^{2}<\infty, \quad
i=1,\ldots,d,
\end{align}
and
\begin{align}\label{eq:2.12}\sum_{k\in\ZZ^{\bar{d}}\setminus\{0\}}\frac{1+\rm{Re}\,\it{q}\left(\frac{\rm{2}\pi
\it{k}}{|\tau|}\right)}{\left|q\left(\frac{2\pi
k}{|\tau|}\right)\right|^{2}}|\widehat{\mathcal{A}^{b}w^{i}}(k)||\widehat{\mathcal{A}^{b}w^{j}}(-k)|<\infty,\quad i,j=1,\ldots,d.\end{align}
Then, for all $i,j=1,\ldots,d$, we have
\begin{align*}\tilde{C}_{ij}&=\lim_{n\longrightarrow\infty}\frac{1}{
n|\tau|
t}\int_{[0,\tau]}\mathbb{E}^{x_\tau}\left[\int_0^{nt}\mathcal{A}^{b}w^{i}(L_s)ds\int_0^{nt}\mathcal{A}^{b}w^{j}(L_r)dr\right]dx_\tau\\&=\frac{1}{|\tau|}\int_{[0,\tau]}w^{i}(x_\tau)\mathcal{A}^{b}w^{j}(x_\tau)dx_\tau\\&=2\sum_{k\in\ZZ^{\bar{d}}\setminus\{0\}}\frac{\rm{Re}\,\it{q}\left(\frac{\rm{2}\pi
\it{k}}{|\tau|}\right)}{\left|q\left(\frac{2\pi
k}{|\tau|}\right)\right|^{2}}\widehat{\mathcal{A}^{b}w^{i}}(k)\widehat{\mathcal{A}^{b}w^{j}}(-k)\\&=2\sum_{k\in\ZZ^{\bar{d}}}\rm{Re}\,\it{q}\left(\frac{\rm{2}\pi
\it{k}}{|\tau|}\right)\hat{w}^{i}(k)\hat{w}^{j}(-k).\end{align*}
\end{proposition}
\begin{proof} First, we prove that
$$\sum_{k\in\ZZ^{\bar{d}}\setminus\{0\}}\frac{\rm{Re}\,\it{q}\left(\frac{\rm{2}\pi
\it{k}}{|\tau|}\right)}{\left|q\left(\frac{\rm{2}\pi
\it{k}}{|\tau|}\right)\right|^{2}}\widehat{\mathcal{A}^{b}w^{i}}(k)\widehat{\mathcal{A}^{b}w^{j}}(-k)=\sum_{k\in\ZZ^{\bar{d}}}^{\infty}\rm{Re}\,\it{q}\left(\frac{\rm{2}\pi
\it{k}}{|\tau|}\right)\hat{w}^{i}(k)\hat{w}^{j}(-k).$$ Because of
$\tau$-periodicity, from Proposition \ref{p2.3},  we have
\begin{align*}\mathcal{A}^{b}w^{i}(x)=&\langle b,\nabla w^{i}(x)\rangle+
\frac{1}{2}\rm{div}\it{c}\nabla w^{i}(x)\nonumber\\&
+\int_{\R^{\bar{d}}}\left(w^{i}(y+x)-w^{i}(x)-\langle y,\nabla
w^{i}(x)\rangle1_{\{z:|z|\leq1\}}(y)\right)\nu(dy).\end{align*} Now,
by using
the assumption \eqref{p2.6a1}
 and the facts that
\begin{align}\label{eq2.13}\frac{\partial w^{i}}{\partial x_p}(x)=\frac{2\pi
i}{|\tau|}\sum_{k\in\ZZ^{\bar{d}}}k_p\hat{w}^{i}(k)e^{i\frac{2\pi
\langle k,x\rangle}{|\tau|}}\quad\textrm{and}\quad
\frac{\partial^{2} w^{i}}{\partial x_p\partial
x_q}(x)=-\frac{4\pi^{2} }{|\tau|^{2}}\sum_{k\in\ZZ^{\bar{d}}}k_p
k_q\hat{w}^{i}(k)e^{i\frac{2\pi \langle
k,x\rangle}{|\tau|}},\end{align} for $i=1,\ldots d$ and
$p,q=1,\ldots \bar{d}$, we easily find
\begin{align}\label{eq2.14}\widehat{\mathcal{A}^{b}w^{i}}(k)=-q\left(\frac{2\pi
k}{|\tau|}\right)\hat{w}^{i}(k),\quad k\in\ZZ^{\bar{d}},\
i=1,\ldots, d,\end{align} which proves the claim. Note that
$$\sum_{k\in\ZZ^{\bar{d}}}\rm{Re}\,\it{q}\left(\frac{\rm{2}\pi
\it{k}}{|\tau|}\right)|\hat{w}^{i}(k)||\hat{w}^{j}(-k)|<\infty$$
follows from (\textbf{C2}) and \eqref{p2.6a1}.

\newpage

Next, we prove
\begin{align*}\lim_{n\longrightarrow\infty}&\frac{1}{
n|\tau|
t}\int_{[0,\tau]}\mathbb{E}^{x_\tau}\left[\int_0^{nt}\mathcal{A}^{b}w^{i}(L_s)ds\int_0^{nt}\mathcal{A}^{b}w^{j}(L_r)dr\right]dx_\tau\\&=2\sum_{k\in\ZZ^{\bar{d}}\setminus\{0\}}\frac{\rm{Re}\,\it{q}\left(\frac{\rm{2}\pi
\it{k}}{|\tau|}\right)}{\left|q\left(\frac{2\pi
k}{|\tau|}\right)\right|^{2}}\widehat{\mathcal{A}^{b}w^{i}}(k)\widehat{\mathcal{A}^{b}w^{j}}(-k).\end{align*}
We have
\begin{align*}&\frac{1}{
n|\tau| t
}\int_{[0,\tau]}\mathbb{E}^{x_\tau}\left[\int_0^{nt}\mathcal{A}^{b}w^{i}(L_r)dr\int_0^{nt}\mathcal{A}^{b}w^{j}(L_s)ds\right]dx_\tau\\&=\frac{1}{n|\tau|
t}\mathbb{E}^{0}\left[\int_{[0,\tau]}\int_0^{nt}\int_0^{nt}\mathcal{A}^{b}w^{i}(L_r+x_\tau)\mathcal{A}^{b}w^{j}(L_s+x_\tau)drdsdx_\tau\right]\\&=\frac{1}{n|\tau|
t}\mathbb{E}^{0}\left[\int_{[0,\tau]}\int_0^{nt}\int_0^{nt}\sum_{k,l\in\ZZ^{\bar{d}}}\widehat{\mathcal{A}^{b}w^{i}}(k)\widehat{\mathcal{A}^{b}w^{j}}(l)e^{i\frac{2\pi
\langle k,(x_\tau+ L_r)\rangle}{|\tau|}}e^{i\frac{2\pi \langle
l,(x_\tau+L_s)\rangle}{|\tau|}}drdsdx_\tau\right]\\&=
\frac{1}{nt}\sum_{k\in\ZZ^{\bar{d}}}\widehat{\mathcal{A}^{b}w^{i}}(k)\widehat{\mathcal{A}^{b}w^{j}}(-k)\int_0^{nt}\int_0^{nt}\mathbb{E}^{0}\left[e^{i\frac{2\pi
\langle k, (L_r-L_s)\rangle}{|\tau|}}\right]drds\\&=
\frac{1}{nt}\sum_{k\in\ZZ^{\bar{d}}}\widehat{\mathcal{A}^{b}w^{i}}(k)\widehat{\mathcal{A}^{b}w^{j}}(-k)\left(\int_0^{nt}\int_0^{s}\mathbb{E}^{0}\left[e^{-i\frac{2\pi
\langle k,
L_{s-r}\rangle}{|\tau|}}\right]drds+\int_0^{nt}\int_s^{nt}\mathbb{E}^{0}
\left[e^{i\frac{2\pi \langle k,
L_{r-s}\rangle}{|\tau|}}\right]drds\right)\\&=\frac{1}{nt}\sum_{k\in\ZZ^{\bar{d}}}\widehat{\mathcal{A}^{b}w^{i}}(k)\widehat{\mathcal{A}^{b}w^{j}}(-k)\left(\int_0^{nt}\int_0^{s}e^{-(s-r)q\left(-\frac{2\pi
k}{|\tau|}\right)}drds+
\int_0^{nt}\int_s^{nt}e^{-(r-s)q\left(\frac{2\pi
k}{|\tau|}\right)}drds\right)\\&=
\sum_{k\in\ZZ^{\bar{d}}\setminus\{0\}}\widehat{\mathcal{A}^{b}w^{i}}(k)\widehat{\mathcal{A}^{b}w^{j}}(-k)\left(\frac{1}{q\left(-\frac{2\pi
k}{|\tau|}\right)}+ \frac{e^{-nt\,q\left(-\frac{2\pi
k}{|\tau|}\right)}-1}{nt\,q^{2}\left(-\frac{2\pi
k}{|\tau|}\right)}+\frac{1}{q\left(\frac{2\pi
k}{|\tau|}\right)}+\frac{e^{-nt\,q\left(\frac{2\pi
k}{|\tau|}\right)}-1}{nt\,q^{2}\left(\frac{2\pi
k}{|\tau|}\right)}\right),
\end{align*}
where in the final step we used the fact that
$\widehat{\mathcal{A}^{b}w^{i}}(0)=0$, $i=1,\ldots,d$, that is,
$\int_{[0,\tau]}\mathcal{A}^{b}w^{i}dx=0$, $i=1,\ldots,d$, (see the
proof of Theorem \ref{tm1.1}). Note that the change of orders of
sums and integrals
 is justified by (\textbf{C2}), \eqref{p2.6a1} and (\ref{eq2.14}).
Now, the desired result follows from  (\ref{eq:2.12}) and the
dominated convergence theorem.

Finally, the fact that
$$\tilde{C}_{ij}=\frac{1}{|\tau|}\int_{[0,\tau]}w^{i}(x_\tau)\mathcal{A}^{b}w^{j}(x_\tau)dx_\tau=
2\sum_{k\in\ZZ^{\bar{d}}}\rm{Re}\,\it{q}\left(\frac{\rm{2}\pi
\it{k}}{|\tau|}\right)\hat{w}^{i}(k)\hat{w}^{j}(-k)$$ follows from a
straightforward computation by using \eqref{p2.6a1},  (\ref{eq2.13}) and (\ref{eq2.14}).
\end{proof}

\begin{proposition} \label{p2.7}Let  $\process{L}$ be a $\bar{d}$-dimensional  L\'evy
process with symbol $q(\xi)$ satisfying the condition in \eqref{eq1.17}. Then, $\process{L^{\tau}}$ is ergodic $($with respect to $dx_\tau/|\tau|$$)$.
\end{proposition}
\begin{proof} First, recall that $\process{L^{\tau}}$ is ergodic if, and only if, the only bounded measurable functions satisfying  $$\int_{[0,\tau]}p_\tau(t,x_\tau,dy_\tau)f_\tau(y_\tau)=f_\tau(x_\tau),\quad x_\tau\in[0,\tau],\ t\geq0,$$ are  constant $dx_\tau$-a.s.
Now, by comparing the Fourier coefficients of the left and right hand side in the above relation, we easily see that (\ref{eq1.17}) implies that the above relation can be satisfied only for  constant $dx_\tau$-a.s. functions.
\end{proof}

Now, we prove Theorem \ref{tm1.3}.
\begin{proof}[Proof of Theorem \ref{tm1.3}]
The proof proceeds in four steps.

\textbf{Step 1.} In the first step, we explain  our strategy of the
proof. The idea of the proof is similar as in the proof of Theorem \ref{tm1.1}.
Namely,  again because of  the independence of
$\process{L}$ and $\process{B}$, \cite[Theorem 36.5]{sato-book} and Proposition \ref{p2.3}, in order to prove the relation in  (\ref{eq1.18}),
it suffices to prove that there
exists a Lebesgue measure zero set $B\in\mathcal{B}(\R^{\bar{d}})$,
such that for any initial distribution $\rho(dx)$ of $\process{L}$,
satisfying $\rho(B)=0$, we have
\begin{align}\label{eq8}n^{-1}\int_0^{nt}\mathcal{A}^{b}w^{i}(L_s)ds\stackrel{\hbox{\scriptsize{$\mathbb{P}^{\rho}\,
\textrm{-}\, \textrm{a.s.}$}}}{\xrightarrow{\hspace*{1cm}}}0,\quad
t\geq0,\ i=1,\ldots,d.\end{align}
Recall that
$w^{1},\ldots,w^{d}\in C^{2}_b(\R^{\bar{d}})$ are $\tau$-periodic.
Further, since the driving diffusion with jumps  is a L\'evy process (hence, it has constant coefficients), we again conclude that
for any $\tau$-periodic  $f:\R^{\bar{d}}\longrightarrow\R$, $f(L_t)=f(L_t^{\tau})$, $t\geq0$, and  that $\mathcal{A}^{b}f(x)$ is $\tau$-periodic for any $\tau$-periodic $f\in C^{2}_b(\R^{\bar{d}})$. Thus,  we can again switch from $\process{L}$ to $\process{L^{\tau}}$, which is, by \eqref{eq1.17}, ergodic (with respect to $dx_\tau/|\tau|$). Now,  the limiting behavior
in (\ref{eq8})  will follow by employing Proposition \ref{p2.4} and the Birkhoff ergodic theorem.

In order to prove the limiting behavior
in (\ref{eq1.19}), again because of the
independence of $\process{F}$ and $\process{B}$ and the scaling
property of $\process{B}$, we conclude that
it
suffices to prove that
there
exists a Lebesgue measure zero set $B\in\mathcal{B}(\R^{\bar{d}})$,
such that for any initial distribution $\rho(dx)$ of $\process{L}$,
satisfying $\rho(B)=0$,
\begin{align}\label{eq2.15}\left\{n^{-\frac{1}{2}}\int_0^{nt}v(L_s)ds\right\}_{t\geq0}\stackrel{\hbox{\scriptsize{$\textrm{d}$}}}{\longrightarrow}\process{\tilde{W}}\end{align}
under $\mathbb{P}^{\rho}(d\omega_V)$, where
$v(x)=(\mathcal{A}^{b}w^{1}(x),\ldots,\mathcal{A}^{b}w^{d}(x))$ and
$\process{\tilde{W}}$ is a zero-drift Brownian motion determined by
the covariance matrix $\tilde{C}:=C-\Sigma$ defined in (\ref{eq1.12}).
Now, we again employ
\cite[Theorem VIII.2.17]{jacod}, which states that
that the desired convergence is reduced to  the convergence (in probability) of the  modified characteristics of  $\left\{n^{-1/2}\int_0^{nt}v(L_s)ds\right\}_{t\geq0}$ to the modified characteristics of $\process{\tilde{W}}$.
Hence, we again explicitly compute the modified characteristics of $\left\{n^{-1/2}\int_0^{nt}v(L_s)ds\right\}_{t\geq0}$ (in terms of the L\'evy triplet of $\process{L}$) and, because of the $\tau$-periodicity of the L\'evy triplet of $\process{L}$, we  switch from $\process{L}$ to $\process{L^{\tau}}$ and apply the Birkhoff ergodic theorem. Finally,  to prove that under (\ref{eq1.20}) the  limit
in (\ref{eq2.15}) holds for any initial distribution of $\process{L}$, we consider the $L^{2}$-convergence of the modified characteristics of $\left\{n^{-1/2}\int_0^{nt}v(L_s)ds\right\}_{t\geq0}$ to the modified characteristics of $\process{\tilde{W}}$.

\textbf{Step 2.} In the second step, we prove the limiting behavior
in (\ref{eq8}). First, according to Proposition \ref{p2.4},
we have
$$\mathcal{A}^{b}w^{i}(L_t)=\mathcal{A}^{b}w^{i}(L^{\tau}_t)=\left(\mathcal{A}^{b}w^{i}\right)_\tau(L^{\tau}_t)=\mathcal{A}_\tau^{b}w^{i}_\tau(L^{\tau}_t),\quad
t\geq0,\ i=1,\ldots,d,$$ which yields
$$n^{-1}\int_0^{nt}\mathcal{A}^{b}w^{i}(L_s)ds=n^{-1}\int_0^{nt}\mathcal{A}_\tau^{b}w_\tau^{i}(L^{\tau}_s)ds,\quad
t\geq0,\ i=1,\ldots,d.$$ Further, according to Proposition \ref{p2.7}, the process $\process{L^{\tau}}$ is
ergodic (with respect to $dx_\tau/|\tau|$). Thus, the Birkhoff  ergodic theorem
entails
$$\mathbb{P}^{dx_\tau/|\tau|}\left(\lim_{n\longrightarrow\infty}n^{-1}\int_0^{nt}\mathcal{A}^{b}w^{i}(L_s)ds=|\tau|^{-1}t\int_{[0,\tau]}\mathcal{A}_\tau^{b}w_\tau^{i}(x_\tau)dx_\tau\right)=1,\quad i=1,\ldots,d.$$
Analogously as in the proof of Theorem \ref{tm1.1}, we conclude that
$$|\tau|^{-1}\int_{[0,\tau]}\mathcal{A}_\tau^{b}w_\tau^{i}(x_\tau)dx_\tau=0,\quad i=1,\ldots,d,$$ that is,

$$|\tau|^{-1}\int_{[0,\tau]}\mathbb{P}^{x_\tau}\left(\lim_{n\longrightarrow\infty}n^{-1}\int_0^{nt}\mathcal{A}^{b}w^{i}(L_s)ds=0\right)dx_\tau=1,\quad
i=1,\ldots,d.$$ Therefore, there exists a Lebesgue measure zero set
$B\in \mathcal{B}(\R^{\bar{d}})$ such that
$$\mathbb{P}^{x}\left(\lim_{n\longrightarrow\infty}n^{-1}\int_0^{nt}\mathcal{A}^{b}w^{i}(L_s)ds=0\right)=1,\quad x\in
B^{c},\ i=1,\ldots, d,$$ which proves the desired result.

\textbf{Step 3.} In the third step, we prove the limiting behavior
in (\ref{eq2.15}). We proceed similarly as in the proof of Theorem
\ref{tm1.1}.  Let
$\rho(dx)$ be an arbitrary initial distribution of $\process{L}$.
Then, again by \cite[Proposition 4.1.7]{ethier}, the processes
$$S^{i,n}_t:=n^{-\frac{1}{2}}\int_0^{nt}\mathcal{A}w^{i}(L_s)ds-n^{-\frac{1}{2}}w^{i}(L_{nt})+n^{-\frac{1}{2}}w^{i}(L_0),\quad i=1,\ldots,d,$$
are  $\mathbb{P}^{\rho}$-martingales. Now, by completely the same
arguments as in the proof of Theorem \ref{tm1.1} we deduce that the
semimartingale (modified) characteristics of $\process{S^{n}}$ are
given by
\begin{align*}
B_t^{n}&=0,\\ C^{i,j,n}_t&=n^{-1}\int_0^{nt}\langle \nabla
w^{i}(L_{s-}),
c\nabla w^{j}(L_{s-})\rangle ds,\quad i,j=1,\ldots,d,\\
\tilde{C}^{i,j,n}_t&=n^{-1}\int_0^{nt}\int_{\R^{\bar{d}}}\left(w^{i}(y+L_{s-})-w^{i}(L_{s-})\right)\left(w^{j}(y+L_{s-})-w^{j}(L_{s-})\right)\nu(dy)ds\\
&\ \ \ +n^{-1}\int_0^{nt}\langle \nabla w^{i}(L_{s-}), c\nabla
w^{j}(L_{s-})\rangle ds,\quad i,j=1,\ldots,d,\\
N^{n}(\omega,ds,B)&=\int_{\R^{\bar{d}}}1_{B}\left(n^{-\frac{1}{2}}w(y+L_{s-}(\omega))-n^{-\frac{1}{2}}w(L_{s-}(\omega))\right)\nu\left(dy\right)ds,\quad
B\in\mathcal{B}(\R^{\bar{d}}),
\end{align*}
where $w(x)=(w^{1}(x),\ldots,w^{d}(x))$. Recall that for the
truncation function we again use an arbitrary
$h:\R^{d}\longrightarrow\R^{d}$, such that $h(x)=x$ for all
$|x|\leq2\max_{i\in\{1,\ldots,d\}}||w^{i}||_\infty.$

 Now, according to
\cite[Theorem VIII.2.17]{jacod}, in order to prove that
$$\process{S^{n}}\stackrel{\hbox{\scriptsize{$\textrm{d}$}}}{\longrightarrow}\process{\tilde{W}},$$ under $\mathbb{P}^{\rho}(d\omega_V)$,
it suffices
to show that
\begin{align}\label{eq2.16}\int_0^{nt}\int_{\R^{d}}|g(y)|N^{n}(\omega,ds,dy)\stackrel{\hbox{\scriptsize{$\mathbb{P}^{\rho}\textrm{-a.s.}$}}}{\xrightarrow{\hspace*{1cm}}} 0\quad
\end{align} for all $t\geq0$ and all $g\in C_b(\R^{d})$
vanishing in a neighborhood around the origin,
and
\begin{align}\label{eq2.17}\tilde{C}^{n}_t\stackrel{\hbox{\scriptsize{$\mathbb{P}^{\rho}\textrm{-a.s.}$}}}{\xrightarrow{\hspace*{1cm}}}t\tilde{C}\end{align}
for  all $t\geq0$. The relation in
(\ref{eq2.16}) easily follows from the fact that the function $w(x)$
is bounded and $g(x)$ vanishes in a neighborhood around the origin. Also, note that (\ref{eq2.16}) holds for any initial distribution $\rho(dx)$ of $\process{L}$.
Now, we prove the relation in (\ref{eq2.17}).
Similarly as in the proof of Theorem \ref{tm1.1},  because of $\tau$-periodicity of
all components,
\begin{align*}\tilde{C}^{i,j,n}_t&=n^{-1}\int_0^{nt}\langle \nabla w^{i}(L^{\tau}_{s-}),
c\nabla w^{j}(L^{\tau}_{s-})\rangle ds \\ &\ \ \
+n^{-1}\int_0^{nt}\int_{\R^{\bar{d}}}\left(w^{i}(y+L^{\tau}_{s-})-w^{i}(L^{\tau}_{s-})\right)\left(w^{j}(y+L^{\tau}_{s-})-w^{j}(L^{\tau}_{s-})\right)\nu(dy)ds\end{align*}
for all $i,j=1,\ldots,d$.  Now, by similar arguments as in the first step,
Proposition \ref{p2.7} implies that $\process{L^{\tau}}$ is
ergodic (with respect to $dx_\tau/|\tau|$), hence  the Birkhoff ergodic theorem
entails that
$$\mathbb{P}^{dx_\tau/|\tau|}\left(\lim_{n\longrightarrow\infty}\tilde{C}^{i,j,n}_t=t\tilde{C}_{ij}\right)=|\tau|^{-1}\int_{[0,\tau]}\mathbb{P}^{x_\tau}\left(\lim_{n\longrightarrow\infty}\tilde{C}^{i,j,n}_t=t\tilde{C}_{ij}\right)dx_\tau=1,\quad i,j=1,\ldots,d.$$
Therefore, there exists a Lebesgue measure zero set
$B\in \mathcal{B}(\R^{\bar{d}})$ such that
$$\mathbb{P}^{x}\left(\lim_{n\longrightarrow\infty}\tilde{C}^{n}_t=t\tilde{C}\right)=1,\quad x\in
B^{c},\ i=1,\ldots, d,$$ which together with  \cite[Lemma VI.3.31]{jacod} proves (\ref{eq2.15}), and thus, (\ref{eq1.19}).

\textbf{Step 4.} In the fourth step, we prove that under (\ref{eq1.20}) the  limit
in (\ref{eq2.15}) holds for any initial distribution of $\process{L}$. We again employ \cite[Theorem VIII.2.17]{jacod}. In the third step we derived the semimartingale (modified) characteristics $(B^{n},C^{n},\tilde{C}^{n},N^{n})$ of the semimartingales $\process{S^{n}}$,  $n\geq1$, and proved that for any initial distribution $\rho(dx)$ of $\process{L}$,
$$\int_0^{nt}\int_{\R^{d}}|g(y)|N^{n}(\omega,ds,dy)\stackrel{\hbox{\scriptsize{$\mathbb{P}^{\rho}\textrm{-a.s.}$}}}{\xrightarrow{\hspace*{1cm}}} 0$$ for all $t\geq0$ and all $g\in C_b(\R^{d})$
vanishing in a neighborhood around the origin. Therefore, the desired result will be proven if we show that for any initial distribution $\rho(dx)$ of $\process{L}$,
$$\tilde{C}^{i,j,n}_t\stackrel{\hbox{\scriptsize{$L^{2}(\mathbb{P}^{\rho},\Omega)$}}}{\xrightarrow{\hspace*{1.2cm}}}t\tilde{C}_{ij}$$
for all $i,j=1,\ldots,d$ and all $t\geq0$.   We have
\begin{align*}\mathbb{E}^{\rho}\left[\left(\tilde{C}^{i,j,n}_t-t\tilde{C}_{ij}\right)^{2}\right]=\mathbb{E}^{\rho}\left[\left(\tilde{C}^{i,j,n}_t\right)^{2}\right]-2t\tilde{C}_{ij}\mathbb{E}^{\rho}\left[\tilde{C}^{i,j,n}_t\right]+t^{2}\tilde{C}_{ij}^{2}.\end{align*}
First, we show that
$\lim_{n\longrightarrow\infty}\mathbb{E}^{\rho}\left[\left(\tilde{C}^{i,j,n}_t\right)^{2}\right]=t^{2}\tilde{C}_{ij}^{2}$.
We have
\begin{align}\label{eq2.19}\mathbb{E}^{\rho}\left[\left(\tilde{C}^{i,j,n}_t\right)^{2}\right]=I^{n}_1+I_2^{n}+I_3^{n},\end{align}
where
\begin{align*}I_1^{n}&:=n^{-2}\sum_{k,l,p,q=1}^{\bar{d}}c_{kl}c_{pq}\int_0^{nt}\int_0^{nt}\mathbb{E}^{\rho}\left[\frac{\partial
w^{i}(L_{s-})}{\partial x_k}\frac{\partial w^{j}(L_{s-})}{\partial
x_l}\frac{\partial w^{i}(L_{u-})}{\partial x_p}\frac{\partial
w^{j}(L_{u-})}{\partial x_q}\right]dsdu,\\
I_2^{n}&:=2n^{-2}\sum_{k,l=1}^{\bar{d}}c_{kl}\int_0^{nt}\int_0^{nt}\mathbb{E}^{\rho}\Bigg[\frac{\partial
w^{i}(L_{s-})}{\partial x_k}\frac{\partial w^{j}(L_{s-})}{\partial
x_l}\int_{\R^{\bar{d}}}\left(w^{i}(y+L_{u-})-w^{i}(L_{u-})\right)\\&
\hspace{5cm}\left(w^{j}(y+L_{u-})-w^{j}(L_{u-})\right)\nu(dy)\Bigg]dsdu,
\\ I_3^{n}&:=n^{-2}\int_0^{nt}\int_0^{nt}\int_{\R^{\bar{d}}}\int_{\R^{\bar{d}}}\mathbb{E}^{\rho}\Big[\left(w^{i}(y+L_{s-})-w^{i}(L_{s-})\right)\left(w^{j}(y+L_{s-})-w^{j}(L_{s-})\right)
\\&
\hspace{4cm}\left(w^{i}(z+L_{u-})-w^{i}(L_{u-})\right)\left(w^{j}(z+L_{u-})-w^{j}(L_{u-})\right)\Big]\nu(dy)\nu(dz)dsdu.\end{align*}
Now, by  the same approach as in Proposition \ref{p2.6}, we get
\begin{align*}I_1^{n}&=2n^{-2}\sum_{k,l,p,q=1}^{\bar{d}}c_{kl}c_{pq}\int_0^{nt}\int_0^{u}\mathbb{E}^{\rho}\left[\frac{\partial
w^{i}(L_{s-})}{\partial x_k}\frac{\partial w^{j}(L_{s-})}{\partial
x_l}\frac{\partial w^{i}(L_{u-})}{\partial x_p}\frac{\partial
w^{j}(L_{u-})}{\partial
x_q}\right]dsdu\\&=\frac{2^{5}\pi^{4}}{n^{2}|\tau|^{4}}\sum_{k,l,p,q=1}^{\bar{d}}c_{kl}c_{pq}\int_0^{nt}\int_0^{u}\sum_{a,b,c,d\in\ZZ^{\bar{d}}}a_kb_lc_pd_q\hat{w}^{i}(a)\hat{w}^{j}(b)\hat{w}^{i}(c)\hat{w}^{j}(d)\\
& \hspace{6.7cm} \mathbb{E}^{\rho}\left[e^{i\frac{2\pi\langle
a+b,L_s\rangle}{|\tau|}}e^{i\frac{2\pi\langle
c+d,L_u\rangle}{|\tau|}}\right]dsdu\\&=\frac{2^{5}\pi^{4}}{n^{2}|\tau|^{4}}\sum_{k,l,p,q=1}^{\bar{d}}c_{kl}c_{pq}\int_0^{nt}\int_0^{u}\sum_{a,b,c,d\in\ZZ^{\bar{d}}}a_kb_lc_pd_q\hat{w}^{i}(a)\hat{w}^{j}(b)\hat{w}^{i}(c)\hat{w}^{j}(d)\hat{\rho}\left({\scriptstyle\frac{2\pi(a+b+c+d)}{|\tau|}}\right)\\
& \hspace{6.7cm} e^{-(u-s)q\left(\frac{2\pi
(c+d)}{|\tau|}\right)}e^{-sq\left(\frac{2\pi
(a+b+c+d)}{|\tau|}\right)}dsdu\\&=\frac{2^{4}\pi^{4}t^{2}}{|\tau|^{4}}\sum_{k,l,p,q=1}^{\bar{d}}c_{kl}c_{pq}\sum_{a,c\in\ZZ^{\bar{d}}}a_ka_lc_pc_q\hat{w}^{i}(a)\hat{w}^{j}(-a)\hat{w}^{i}(c)\hat{w}^{j}(-c)\\
&\ \ \
+\frac{2^{5}\pi^{4}}{n^{2}|\tau|^{4}}\sum_{k,l,p,q=1}^{\bar{d}}c_{kl}c_{pq}\int_0^{nt}\int_0^{u}\sum_{{a,b,c,d\in\ZZ^{\bar{d}}}\atop{
a+b\neq0\,\textrm{or}\,
c+d\neq0}}a_kb_lc_pd_q\hat{w}^{i}(a)\hat{w}^{j}(b)\hat{w}^{i}(c)\hat{w}^{j}(d)\hat{\rho}\left({\scriptstyle\frac{2\pi(a+b+c+d)}{|\tau|}}\right)\\&
\hspace{7.7cm}e^{-(u-s)q\left(\frac{2\pi
(c+d)}{|\tau|}\right)}e^{-sq\left(\frac{2\pi
(a+b+c+d)}{|\tau|}\right)}dsdu,\end{align*} where $\hat{\rho}(\xi)$
denotes the characteristic function of the probability measure
$\rho(dx).$ Note that the change of orders of integrations and
summations is justified by (\ref{eq1.20}). Finally, again by
applying (\ref{eq1.20}), it is easy to see that
\begin{align}\label{eq2.20}\lim_{n\longrightarrow\infty}I_1^{n}=\frac{2^{4}\pi^{4}t^{2}}{|\tau|^{4}}\sum_{k,l,p,q=1}^{\bar{d}}c_{kl}c_{pq}\sum_{a,c\in\ZZ^{\bar{d}}}a_ka_lc_pc_q\hat{w}^{i}(a)\hat{w}^{j}(-a)\hat{w}^{i}(c)\hat{w}^{j}(-c).\end{align}
\newpage
Similarly, we have
\begin{align*}I_2^{n}&=4n^{-2}\sum_{k,l=1}^{\bar{d}}c_{kl}\int_0^{nt}\int_0^{u}\mathbb{E}^{\rho}\Bigg[\frac{\partial
w^{i}(L_{s-})}{\partial x_k}\frac{\partial w^{j}(L_{s-})}{\partial
x_l}\int_{\R^{\bar{d}}}\left(w^{i}(y+L_{u-})-w^{i}(L_{u-})\right)\\&
\hspace{5cm}\left(w^{j}(y+L_{u-})-w^{j}(L_{u-})\right)\nu(dy)\Bigg]dsdu\\&=
-\frac{2^{4}\pi^{2}}{n^{2}|\tau|^{2}}\sum_{k,l=1}^{\bar{d}}c_{kl}\int_0^{nt}\int_0^{u}\sum_{a,b,c,d\in\ZZ^{\bar{d}}}a_kb_l\hat{w}^{i}(a)\hat{w}^{j}(b)\hat{w}^{i}(c)\hat{w}^{j}(d)\mathbb{E}^{\rho}\left[e^{i\frac{2\pi\langle
a+b,L_s\rangle}{|\tau|}}e^{i\frac{2\pi\langle
c+d,L_u\rangle}{|\tau|}}\right]\\&\hspace{6cm}\int_{\R^{\bar{d}}}\left(e^{i\frac{2\pi\langle
c,y\rangle}{|\tau|}}-1\right)\left(e^{i\frac{2\pi\langle
d,y\rangle}{|\tau|}}-1\right)\nu(dy)dsdu\\&=-\frac{2^{4}\pi^{2}}{n^{2}|\tau|^{2}}\sum_{k,l=1}^{\bar{d}}c_{kl}\int_0^{nt}\int_0^{u}\sum_{a,b,c,d\in\ZZ^{\bar{d}}}a_kb_l\hat{w}^{i}(a)\hat{w}^{j}(b)\hat{w}^{i}(c)\hat{w}^{j}(d)\hat{\rho}\left({\scriptstyle\frac{2\pi(a+b+c+d)}{|\tau|}}\right)e^{-(u-s)q\left(\frac{2\pi(c+d)}{|\tau|}\right)}\\&\hspace{4.9cm}e^{-sq\left(\frac{2\pi(a+b+c+d)}{|\tau|}\right)}\int_{\R^{\bar{d}}}\left(e^{i\frac{2\pi\langle
c,y\rangle}{|\tau|}}-1\right)\left(e^{i\frac{2\pi\langle
d,y\rangle}{|\tau|}}-1\right)\nu(dy)dsdu\\&=\frac{2^{4}\pi^{2}t^{2}}{|\tau|^{2}}\sum_{k,l=1}^{\bar{d}}c_{k,l}\sum_{a,c\in\ZZ^{\bar{d}}}a_ka_l\hat{w}^{i}(a)\hat{w}^{j}(-a)\hat{w}^{i}(c)\hat{w}^{j}(-c)\int_{\R^{\bar{d}}}\left(1-\cos\left({\scriptstyle\frac{2\pi\langle
c,y\rangle}{|\tau|}}\right)\right)\nu(dy)\\ &\ \ \
-\frac{2^{4}\pi^{2}}{n^{2}|\tau|^{2}}\sum_{k,l=1}^{\bar{d}}c_{k,l}\int_0^{nt}\int_0^{u}\sum_{{a,b,c,d\in\ZZ^{\bar{d}}}\atop{
a+b\neq0\,\textrm{or}\,
c+d\neq0}}a_kb_l\hat{w}^{i}(a)\hat{w}^{j}(b)\hat{w}^{i}(c)\hat{w}^{j}(d)\hat{\rho}\left({\scriptstyle\frac{2\pi(a+b+c+d)}{|\tau|}}\right)\\&\hspace{2.5cm}e^{-(u-s)q\left(\frac{2\pi(c+d)}{|\tau|}\right)}e^{-sq\left(\frac{2\pi(a+b+c+d)}{|\tau|}\right)}\int_{\R^{\bar{d}}}\left(e^{i\frac{2\pi\langle
c,y\rangle}{|\tau|}}-1\right)\left(e^{i\frac{2\pi\langle
d,y\rangle}{|\tau|}}-1\right)\nu(dy)dsdu.
\end{align*}
Again, by applying (\ref{eq1.20}), we get
\begin{align}\label{eq2.21}\lim_{n\longrightarrow\infty}I_2^{n}=\frac{2^{4}\pi^{2}t^{2}}{|\tau|^{2}}\sum_{k,l=1}^{\bar{d}}c_{k,l}\sum_{a,c\in\ZZ^{\bar{d}}}&a_ka_l\hat{w}^{i}(a)\hat{w}^{j}(-a)\hat{w}^{i}(c)\hat{w}^{j}(-c)\nonumber\\&\int_{\R^{\bar{d}}}\left(1-\cos\left({\scriptstyle\frac{2\pi\langle
c,y\rangle}{|\tau|}}\right)\right)\nu(dy).\end{align} Finally, we
have
\begin{align*}I_3^{n}&=2n^{-2}\int_0^{nt}\int_0^{u}\int_{\R^{\bar{d}}}\int_{\R^{\bar{d}}}\mathbb{E}^{\rho}\Big[\left(w^{i}(y+L_{s-})-w^{i}(L_{s-})\right)\left(w^{j}(y+L_{s-})-w^{j}(L_{s-})\right)
\\&
\hspace{4cm}\left(w^{i}(z+L_{u-})-w^{i}(L_{u-})\right)\left(w^{j}(z+L_{u-})-w^{j}(L_{u-})\right)\Big]\nu(dy)\nu(dz)dsdu\\&=
2n^{-2}\int_0^{nt}\int_0^{u}\sum_{a,b,c,d\in\ZZ^{\bar{d}}}\hat{w}^{i}(a)\hat{w}^{j}(b)\hat{w}^{i}(c)\hat{w}^{j}(d)\mathbb{E}^{\rho}\left[e^{i\frac{2\pi\langle
a+b,L_s\rangle}{|\tau|}}e^{i\frac{2\pi\langle
c+d,L_u\rangle}{|\tau|}}\right]\\&\hspace{2cm}\int_{\R^{\bar{d}}}\int_{\R^{\bar{d}}}\left(e^{i\frac{2\pi\langle
a,y\rangle}{|\tau|}}-1\right) \left(e^{i\frac{2\pi\langle
b,y\rangle}{|\tau|}}-1\right)\left(e^{i\frac{2\pi\langle
c,z\rangle}{|\tau|}}-1\right)\left(e^{i\frac{2\pi\langle
d,z\rangle}{|\tau|}}-1\right)\nu(dy)\nu(dz)dsdu
\end{align*}\newpage
\begin{align*}&=2n^{-2}\int_0^{nt}\int_0^{u}\sum_{a,b,c,d\in\ZZ^{\bar{d}}}\hat{w}^{i}(a)\hat{w}^{j}(b)\hat{w}^{i}(c)\hat{w}^{j}(d)\hat{\rho}\left({\scriptstyle\frac{2\pi(a+b+c+d)}{|\tau|}}\right)e^{-(u-s)q\left(\frac{2\pi(c+d)}{|\tau|}\right)}e^{-sq\left(\frac{2\pi(a+b+c+d)}{|\tau|}\right)}\\&\hspace{1.8cm}\int_{\R^{\bar{d}}}\int_{\R^{\bar{d}}}\left(e^{i\frac{2\pi\langle
a,y\rangle}{|\tau|}}-1\right)\left(e^{i\frac{2\pi\langle
b,y\rangle}{|\tau|}}-1\right)\left(e^{i\frac{2\pi\langle
c,z\rangle}{|\tau|}}-1\right)\left(e^{i\frac{2\pi\langle
d,z\rangle}{|\tau|}}-1\right)\nu(dy)\nu(dz)dsdu\\&=4t^{2}\sum_{a,c\in\ZZ^{\bar{d}}}\hat{w}^{i}(a)\hat{w}^{j}(-a)\hat{w}^{i}(c)\hat{w}^{j}(-c)\int_{\R^{\bar{d}}}\int_{\R^{\bar{d}}}\left(1-\cos\left({\scriptstyle\frac{2\pi\langle
a,y\rangle}{|\tau|}}\right)\right)\left(1-\cos\left({\scriptstyle\frac{2\pi\langle
c,z\rangle}{|\tau|}}\right)\right)\nu(dy)\nu(dz)\\&\ \ \
+2n^{-2}\int_0^{nt}\int_0^{u}\sum_{{a,b,c,d\in\ZZ^{\bar{d}}}\atop{
a+b\neq0\,\textrm{or}\,
c+d\neq0}}\hat{w}^{i}(a)\hat{w}^{j}(b)\hat{w}^{i}(c)\hat{w}^{j}(d)\hat{\mu}\left({\scriptstyle\frac{2\pi(a+b+c+d)}{|\tau|}}\right)e^{-(u-s)q\left(\frac{2\pi(c+d)}{|\tau|}\right)}e^{-sq\left(\frac{2\pi(a+b+c+d)}{|\tau|}\right)}\\&\hspace{1.6cm}\int_{\R^{\bar{d}}}\int_{\R^{\bar{d}}}\left(e^{i\frac{2\pi\langle
a,y\rangle}{|\tau|}}-1\right)\left(e^{i\frac{2\pi\langle
b,y\rangle}{|\tau|}}-1\right)\left(e^{i\frac{2\pi\langle
c,z\rangle}{|\tau|}}-1\right)\left(e^{i\frac{2\pi\langle
d,z\rangle}{|\tau|}}-1\right)\nu(dy)\nu(dz)dsdu.
\end{align*}
Again, (\ref{eq1.20}) implies that
\begin{align}\label{eq2.22}\lim_{n\longrightarrow\infty}I_3^{n}=4t^{2}\sum_{a,c\in\ZZ^{\bar{d}}}&\hat{w}^{i}(a)\hat{w}^{j}(-a)\hat{w}^{i}(c)\hat{w}^{j}(-c)\nonumber\\& \int_{\R^{\bar{d}}}\int_{\R^{\bar{d}}}\left(1-\cos\left({\scriptstyle\frac{2\pi\langle
a,y\rangle}{|\tau|}}\right)\right)\left(1-\cos\left({\scriptstyle\frac{2\pi\langle
c,z\rangle}{|\tau|}}\right)\right)\nu(dy)\nu(dz).\end{align} Now, by
putting together (\ref{eq2.19}), (\ref{eq2.20}), (\ref{eq2.21}) and
(\ref{eq2.22}), Proposition \ref{p2.6} implies
$$\lim_{n\longrightarrow\infty}\mathbb{E}^{\rho}\left[\left(\tilde{C}^{i,j,n}_t\right)^{2}\right]=t^{2}\tilde{C}_{ij}^{2}.$$
In completely the same way we get
$$\lim_{n\longrightarrow\infty}\mathbb{E}^{\rho}\left[\tilde{C}^{i,j,n}_t\right]=t\tilde{C}_{ij}.$$
Thus,
$$\tilde{C}^{i,j,n}_t\stackrel{\hbox{\scriptsize{$L^{2}(\mathbb{P}^{\rho},\Omega)$}}}{\xrightarrow{\hspace*{1.2cm}}}t\tilde{C}_{ij},\quad i,j=1,\ldots,d,$$
that is, for any initial distribution $\rho(dx)$ of $\process{L}$,
$$\process{S^{n}}\stackrel{\hbox{\scriptsize{$\textrm{d}$}}}{\longrightarrow}\process{\tilde{W}},$$ under $\mathbb{P}^{\rho}(d\omega_V)$. Finally, since the function
$w(x)$ is bounded, \cite[Lemma VI.3.31]{jacod} again implies the
convergence in (\ref{eq2.15}), and thus, in (\ref{eq1.19}).
\end{proof}

\subsection{Comments on the Condition in \eqref{eq1.17}} \label{sec-Comment}

\ \ \ \ In connection to Proposition \ref{p2.7}, note that  if (\ref{eq1.17}) is not satisfied for some $k_0\neq0$, then we cannot automatically conclude that $\process{L^{\tau}}$ is not ergodic. For example, take a one-dimensional L\'evy process $\process{L}$ with  symbol of the form $q(\xi)=ib\xi$, $b\neq0$. On the other hand, in the dimension $\bar{d}\geq2$ or when $b=0$ (in any dimension), $\process{L^{\tau}}$ is not ergodic.
\begin{proposition} \label{p2.8}Let  $\process{L}$ be a $\bar{d}$-dimensional  L\'evy
process with symbol $q(\xi)$  not satisfying the condition in \eqref{eq1.17}. Then, $\process{L^{\tau}}$ is not strongly ergodic $($with respect to $dx_\tau/|\tau|$$)$.
\end{proposition}
\begin{proof} By the assumption, there exists $k_0\in\ZZ^{\bar{d}}$, $k_0\neq0$, such that ${\rm Re}\, q(2\pi {\it k}_0/|\tau|)=0$.
Hence, for this  $k_0\in\ZZ^{\bar{d}}$, we have
 $$\mathbb{E}^{0}\left[e^{i\left\langle\frac{2\pi k_0}{|\tau|},L_t\right\rangle}\right]=e^{i\left\langle\frac{2\pi k_0}{|\tau|},tx_0\right\rangle}$$ for some $x_0\in\R^{\bar{d}}$. This yields
$$\mathbb{E}^{0}\left[\cos{\scriptstyle\left\langle\frac{2\pi k_0}{|\tau|},L_t-tx_0\right\rangle}\right]=\int_{\R^{\bar{d}}}\cos{\scriptstyle\left\langle\frac{2\pi k_0}{|\tau|},y-tx_0\right\rangle}p(t,0,dy)=1.$$
Thus,  $p(t,0,dy)$ is supported on the set $\{y\in\R^{\bar{d}}:\langle k_0,y-tx_0\rangle=l|\tau|,\ l\in\ZZ\}$, $t>0$.
In particular,  $p(t,0,dy)$ is singular with respect to $dx$, which proves the claim.
\end{proof}

Further, the condition in $(\ref{eq1.17})$ is also not equivalent with the strong ergodicity  of $\process{L^{\tau}}$. For example, let $\process{L}$ be a one-dimensional L\'evy process with symbol of the form $q(\xi)=2(1-\cos(\kappa\xi))$ or, equivalently, with the L\'evy triplet $(0,0,\delta_{-\kappa}(dy)+\delta_\kappa(dy))$, where $\kappa>0$ is such that $\kappa/\tau\notin\QQ.$ However, as a direct consequence of Proposition \ref{p2.8} we get  that condition (\textbf{C5}) automatically  implies the relation in (\ref{eq1.17}).
\begin{proposition}\label{p2.9} Let  $\process{L}$ be a $\bar{d}$-dimensional L\'evy process with symbol $q(\xi)$ and L\'evy triplet $(b,0,\nu(dy))$.
\begin{itemize}
  \item [$(i)$] If $\nu(dy)=0$, then for any $\beta>0$, any $\tau$-periodic $f\in C_b^{2}(\R^{\bar{d}})$ and any initial distribution $\rho(dx)$ of  $\process{L}$,
$$n^{-\beta}\int_0^{nt}\mathcal{A}^{b}f(L_s)ds\stackrel{\hbox{\scriptsize{$\mathbb{P}^{\rho}\textrm{-a.s.}$}}}{\xrightarrow{\hspace*{1cm}}}0.$$
\item [$(ii)$] If $\bar{d}=1$ and if \eqref{eq1.17} is not satisfied for some $k_0\neq0$, then
for any $\beta>0$, any $\tau/|k_0|$-periodic $f\in C_b^{2}(\R)$ and any initial distribution $\rho(dx)$ of  $\process{L}$,
$$n^{-\beta}\int_0^{nt}\mathcal{A}^{b}f(L_s)ds\stackrel{\hbox{\scriptsize{$\mathbb{P}^{\rho}\textrm{-a.s.}$}}}{\xrightarrow{\hspace*{1cm}}}0.$$
\end{itemize}
\end{proposition}
\begin{proof}
\begin{itemize}
  \item [(i)] Let $f\in C_b^{2}(\R^{\bar{d}})$ and  $\rho(dx)$ be an arbitrary  $\tau$-periodic function and an arbitrary initial distribution  of  $\process{L}$, respectively. Then, for any $\beta>0$, we have
\begin{align*}\lim_{n\longrightarrow\infty}n^{-\beta}\int_0^{nt}\mathcal{A}^{b}f(L_s)ds&=\lim_{n\longrightarrow\infty}n^{-\beta}\int_0^{nt}\langle b,\nabla
f(L_0+bs)\rangle ds\\&=
\lim_{n\longrightarrow\infty}n^{-\beta}\int_0^{nt}\frac{\partial
f(L_0+bs)}{\partial
s}ds\\&=\lim_{n\longrightarrow\infty}n^{-\beta}(f(L_0+bnt)-f(L_0))\\&=0,\quad
\mathbb{P}^{\rho}\textrm{-a.s.}\end{align*}
  \item [(ii)] First, similarly as before, we conclude that $\nu(dy)$ is supported on $S:=\{\tau l/k_0:\ l\in\ZZ\}.$ This yields that
$$\mathcal{A}^{b}g(x)=bg'(x)+\int_{\R}(g(y+x)-g(x))\nu(dy),\quad g\in C_b^{2}(\R),$$ $\nu(\R)<\infty$ and
$L_t=L_0+bt+S_{N_t},$ $ t\geq0,$ where $\chain{S}$ is a random walk on $S$ with the jump distribution
$\mathbb{P}^{0}(S_1\in dy):=\nu(dy)/\nu(\R)$ and $\process{N}$ is a Poisson process with parameter $\nu(\R)$ independent of $\chain{S}$.
Now, let $f\in C_b^{2}(\R)$ be an arbitrary $\tau/|k_0|$ periodic function. Then, for any $\beta>0$, we have
\begin{align*}n^{-\beta}\int_0^{nt}\mathcal{A}^{b}f(L_s)ds=&n^{-\beta}\int_0^{nt}bf'(L_0+bs+S_{N_s})ds\\&
+n^{-\beta}\int_0^{nt}\int_{\R}(f(y+L_0+bs+S_{N_s})-f(L_0+bs+S_{N_s}))\nu(dy)ds\\=&n^{-\beta}\int_0^{nt}bf'(L_0+bs)ds\\=&n^{-\beta}\int_0^{nt}\frac{\partial
f(L_0+bs)}{\partial
s}ds\\=&n^{-\beta}(f(L_0+bnt)-f(L_0)),\end{align*} where in the second step we used the facts that $f(x)$ is $\tau/|k_0|$-periodic and $\chain{S}$ and $\nu(dy)$ live on $S$. Now, by letting $n\longrightarrow\infty$, the desired result follows.\end{itemize}
\end{proof}

\section{Discussions on the Cases with General Diffusions with Jumps} \label{sec-Discussions}

\ \ \ \ In Theorems \ref{tm1.1}, \ref{tm1.2} and \ref{tm1.3}, we have shown the LLN and CLT for the process $\process{X}$ in \eqref{eq1.2}  driven by the process $\process{F}$, a diffusion with jumps, satisfying the sets of assumed conditions, particularly the ergodicity property.
In this section, we discuss the (strong) ergodicity property of general diffusions with jumps and the limiting behaviors in \eqref{eq1.3} and \eqref{eq1.4}  when the velocity field $\{V(t,x)\}_{t\geq0,\, x\in\R^{\bar{d}}}$  is governed by  general, not necessarily ergodic, diffusions with jumps.

In the proofs of Theorems \ref{tm1.1} and \ref{tm1.3}, the most crucial ingredient was the $\tau$-periodicity
of a driving diffusion with jumps $\process{F}$ and velocity function $v(x)$. By having this property we were able
to switch to a (strongly) ergodic Markov process $\process{F^{\tau}}$ on a compact space
$[0,\tau]$, satisfying
$v(F_t)=v(F_t^{\tau})$, and deduced the limiting behaviors in (\ref{eq1.3}) and (\ref{eq1.4}). In a general situation, when  $\process{F}$ or
$v(x)$ are not $\tau$-periodic, we cannot perform a similar trick,
and unlike in the $\tau$-periodic case, the
(strong) ergodicity  strongly depends on the dimension of the state space.
Let us be more precise.
First, recall that  a progressively measurable  strong
Markov process $\process{M}$  on the state
space $(\R^{\bar{d}},\mathcal{B}(\R^{\bar{d}}))$, $\bar{d}\geq1$, is called
\begin{enumerate}
  \item [(i)] \emph{irreducible} if there exists a $\sigma$-finite measure $\varphi(dy)$ on
$\mathcal{B}(\R^{\bar{d}})$ such that whenever $\varphi(B)>0$ we have
$\int_0^{\infty}\mathbb{P}^{x}(M_t\in B)dt>0$ for all $x\in\R^{\bar{d}}$;
  \item [(ii)] \emph{recurrent} if it is
                      $\varphi$-irreducible and if $\varphi(B)>0$ implies $\int_{0}^{\infty}\mathbb{P}^{x}(M_t\in B)dt=\infty$ for all
                      $x\in\R^{\bar{d}}$;
\item [(iii)] \emph{Harris recurrent} if it is $\varphi$-irreducible and if $\varphi(B)>0$ implies $\mathbb{P}^{x}(\tau_B<\infty)=1$ for all
                      $x\in\R^{\bar{d}}$, where $\tau_B:=\inf\{t\geq0:M_t\in
                      B\}$ ;
 \item [(iv)] \emph{transient} if it is $\varphi$-irreducible
                       and if there exists a countable
                      covering of $\R^{\bar{d}}$ with  sets
$\{B_j\}_{j\in\N}\subseteq\mathcal{B}(\R^{\bar{d}})$, such that for each
$j\in\N$ there is a finite constant $c_j\geq0$ such that
$\int_0^{\infty}\mathbb{P}^{x}(M_t\in B_j)dt\leq c_j$ holds for all
$x\in\R^{\bar{d}}$.
\end{enumerate}
Let us remark that if $\{M_t\}_{t\geq0}$ is a $\varphi$-irreducible
Markov process, then the irreducibility measure $\varphi(dy)$ can be
maximized, that is, there exists a unique ``maximal" irreducibility
measure $\psi(dy)$ such that for any measure $\bar{\varphi}(dy)$,
$\{M_t\}_{t\geq0}$ is $\bar{\varphi}$-irreducible if, and only if,
$\bar{\varphi}\ll\psi$ (see \cite[Theorem 2.1]{tweedie-mproc}).
According to this, from now on, when we refer to irreducibility
measure we actually refer to the maximal irreducibility measure. In
the sequel, we consider  only the so-called open set irreducible
Markov processes, that is,  we consider only $\psi$-irreducible
Markov processes whose maximal irreducibility measure $\psi(dy)$
satisfies the following \emph{open-set irreducibility} condition
\begin{description}
  \item[(C6)] $\psi(O)>0$ for every open set $O\subseteq\R^{\bar{d}}$.
\end{description}
Obviously, the Lebesgue measure $dx$ satisfies condition
(\textbf{C6}) and a Markov process $\process{M}$ will be
$dx$-irreducible if $\mathbb{P}^{x}(M_t\in B)>0$ for all $t>0$
and all $x\in\R^{\bar{d}}$ whenever $B\in\mathcal{B}(\R^{\bar{d}})$ has  positive Lebesgue measure. In particular, the process
$\process{M}$ will be $dx$-irreducible if the transition kernel
$\mathbb{P}^{x}(M_t\in dy)$ possesses a transition density function $p(t,x,y)$,
such that $p(t,x,y)>0$ for all $t>0$ and all $x,y\in\R^{\bar{d}}.$
In Remark \ref{rm2} we have  commented that the question of the strict positivity  of $p(t,x,y)$ of general diffusions with jumps is a  non-trivial problem. The best we were able to obtain is given in the following proposition, which, regardless the possible lack of the strict positivity of  $p(t,x,y)$, shows the
Lebesgue irreducibility property of a class of diffusions with jumps.
\begin{proposition}\label{p2.2}
Let $\process{F}$ be a  diffusion with jumps with symbol $q(x,\xi)$
satisfying \eqref{eq3.01} and
\begin{align}\label{eq6}\mathbb{E}^{x}\left[e^{i\langle \xi,F_t-x\rangle}\right]={\rm Re}\,\mathbb{E}^{x}\left[e^{i\langle \xi,F_t-x\rangle}\right],\quad x,\xi\in\R^{\bar{d}}.\end{align}
Then, $\process{F}$ possesses a transition
density function $p(t,x,y)$, such that for every
$t_0>0$ and every $n\geq1$ there exists $\varepsilon(t_0)>0$ such that $p(t,x,y+x)>0$
for all $t\in[nt_0,n(t_0+1)]$, all $x\in\R^{\bar{d}}$ and all
$|y|<n\varepsilon(t_0)$.
\end{proposition}
\begin{proof}
First, according to \cite[Theorem 2.1]{rene-wang-feller},  the  condition in \eqref{eq6} implies  that $q(x,\xi)=\rm{Re}\,\it{q}(x,\xi)$, that is, $b(x)=0$
and $\nu(x,dy)$ are symmetric measures for all $x\in\R^{\bar{d}}$. Consequently, the condition in \eqref{eq3.02} trivially holds  true.
Thus, $\process{F}$
possesses a transition density function $p(t,x,y)$
which is given by (\ref{eq2.5}), and,  for every $t>0$ and
every $x,y\in\R^{\bar{d}}$, we have
\begin{align*}|p(t,x,y+x)-p(t,x,x)|&=(2\pi)^{-\bar{d}}\left|\int_{\R^{\bar{d}}}\left(1-e^{-i\langle\xi, y\rangle}\right)\mathbb{E}^{x}\left[e^{i\langle\xi,(F_t-x)\rangle}\right]d\xi\right|\\&\leq
(2\pi)^{-\bar{d}}\int_{\R^{\bar{d}}}\left|1-e^{-i\langle\xi,y\rangle}\right|\exp\left[-\frac{t}{16}\inf_{x\in\R^{\bar{d}}}\it{q}(x,\rm{2}\xi)\right]d\xi.\end{align*}
Now, by \eqref{eq3.01} and the dominated convergence theorem, we conclude that for
every $t_0>0$ the continuity of the function $y\longmapsto p(t,x,y)$
at $x$ is uniformly for all $t\geq t_0$ and all $x\in\R^{\bar{d}}.$ Next,
by applying \cite[Theorem 2.1]{rene-wang-feller} (under \eqref{eq6}), R. L. Schilling and  J. Wang (personal communication) obtained
$$\inf_{x\in\R^{\bar{d}}}\mathbb{E}^{x}\left[e^{i\langle\xi,(F_t-x)\rangle}\right]\geq\frac{1}{2}\exp\left[-4t\sup_{|\eta|\leq|\xi|}\sup_{x\in\R^{\bar{d}}}q(x,\eta)\right],\quad
t>0,\ \xi\in\R^{\bar{d}}.$$ Hence, for every $t_0>0$,
$$p(t,x,x)=(2\pi)^{-\bar{d}}\int_{\R^{\bar{d}}}\mathbb{E}^{x}\left[e^{i\langle\xi,(F_t-x)\rangle}\right]d\xi\geq\frac{1}{4\pi}\int_{\R^{\bar{d}}}\exp\left[-4(t_0+1)\sup_{|\eta|\leq|\xi|}\sup_{x\in\R^{\bar{d}}}q(x,\eta)\right]d\xi>0$$
uniformly for all $t\in[t_0,t_0+1]$ and all $x\in\R^{\bar{d}}$. According
to this, there exists $\varepsilon(t_0)>0$ such that $p(t,x,y+x)>0$
for all $t\in[t_0,t_0+1]$, all $x\in\R^{\bar{d}}$ and all
$|y|<\varepsilon(t_0)$. Now, for any $n\geq1$, by the
Chapman-Kolmogorov equation, we have that  $p(t,x,y+x)>0$ for all
$t\in[nt_0,n(t_0+1)]$, all $x\in\R^{\bar{d}}$ and all $|y|<n\varepsilon(t_0)$,
which proves the desired result.
\end{proof}

  Further, it is well known that
every $\psi$-irreducible Markov process is either recurrent or
transient (see \cite[Theorem 2.3]{tweedie-mproc}) and, clearly,
every Harris recurrent Markov process is recurrent,  but in general,
these two properties are not equivalent. They differ on the set of
the irreducibility measure zero (see \cite[Theorem
2.5]{tweedie-mproc}). However,
 for a diffusion with jumps satisfying
condition (\textbf{C6}) these two
 properties are equivalent (see  \cite[Proposition
 2.1]{sandric-tams}).

Next, it is shown in \cite[Theorem 2.6]{tweedie-mproc} that if $\process{M}$ is
a recurrent process, then there exists a unique (up to constant
multiples) invariant measure $\pi(dx)$. If the invariant measure is
finite, then it may be normalized to a probability measure.  If
$\process{M}$ is (Harris) recurrent with finite invariant measure
$\pi(dx)$, then $\process{M}$ is called \emph{positive $($Harris$)$
recurrent}; otherwise it is called \emph{null $($Harris$)$ recurrent}. One would expect
that every positive (Harris) recurrent process  is strongly ergodic,
but in general this is not true (see \cite{meyn-tweedie-II}). In the
case of an open-set irreducible diffusion with jumps $\process{F}$,
these two properties coincide.
In particular, for this class of
processes, ergodicity coincides with strong ergodicity. Indeed,
according to \cite[Theorem 6.1]{meyn-tweedie-II} and \cite[Theorem
3.3]{rene-wang-feller} it suffices to show that if
 $\process{F}$ possesses an invariant probability measure $\pi(dx)$, then
it is recurrent. Assume that $\process{F}$ is transient. Then there
exists a countable
                      covering of $\R^{\bar{d}}$ with  sets
$\{B_j\}_{j\in\N}\subseteq\mathcal{B}(\R^{\bar{d}})$, such that for each
$j\in\N$ there is a finite constant $M_j\geq0$ such that
$\int_0^{\infty}p^{t}(x,B_j)dt\leq M_j$ holds for all $x\in\R^{\bar{d}}$.
Let $t>0$ be arbitrary. Then for each $j\in\N$ we have
$$t\pi(B_j)=\int_0^{t}\int_{\R^{d}}p^{s}(x,B_j)\pi(dx)ds\leq M_j.$$ By
letting $t\longrightarrow\infty$ we get that $\pi(B_j)=0$ for all
$j\in\N$, which is impossible.

Now, we can conclude that a $\bar{d}$-dimensional, $\bar{d}\geq3$, open-set irreducible diffusion with jumps is never ergodic.
Indeed, as we commented,  if an invariant
probability measure of  an irreducible Markov process  exists, then it must be  recurrent. But  the recurrence and
transience of open set irreducible diffusions with jumps, similarly
as of L\'evy processes, depends on the dimension of the state space.
More precisely, according to \cite[Theorem 2.8]{sandric-tams}, every
truly $\bar{d}$-dimensional, $\bar{d}\geq3$,  open set irreducible diffusions with jumps is always transient. In particular, $\bar{d}$-dimensional,
$\bar{d}\geq3$,  diffusions are transient. On the other hand,  one-dimensional
and two-dimensional symmetric  diffusions are recurrent
(see \cite[Theorem 2.9]{sandric-tams}). Also $\bar{d}$-dimensional,
$\bar{d}\geq2$, stable-like processes are always transient (see
\cite[Theorem 2.10 and Corollary 3.3]{sandric-tams}). For
conditions for recurrence and transience of one-dimensional
stable-like processes,  see \cite{bjoern-overshoot}, \cite{franke-periodic, franke-periodicerata}, \cite{sandric-spa},
\cite{sandric-rectrans},
\cite{sandric-ergodic}, \cite{sandric-tams} and \cite{sandric-periodic}.
 For sufficient conditions for
ergodicity of  diffusions, see \cite{bhat-erg-1, bhat-erg-1-er}, \cite{bhat-erg-2}. For
sufficient  conditions for strong ergodicity of one-dimensional stable-like processes and diffusion with jumps,
see \cite{sandric-spa} and \cite{wang-ergodic}, respectively.  A necessary and sufficient condition for the existence
of an invariant probability measure $\pi(dx)$ of a $\bar{d}$-dimensional diffusion with jumps
$\process{F}$  with symbol
$q(x,\xi)$, for which $C_c^{\infty}(\R^{\bar{d}})$ is an operator core of the corresponding Feller generator (that is,
$\mathcal{A}^{\infty}$ is the only extension of $\mathcal{A}^{\infty}|_{C_c^{\infty}(\R^{\bar{d}})}$ on
$\mathcal{D}_{\mathcal{A}^{\infty}}$), has been given in \cite[Theorems 3.1 and 4.1]{sch-beh}
and it reads as follows
$$\int_{\R^{\bar{d}}}e^{i\langle\xi,x\rangle}q(x,\xi)\pi(dx)=0,\quad
 \xi\in\R^{\bar{d}}.$$
Also, let us remark that a (nontrivial) L\'{e}vy
process is never (strongly) ergodic since it cannot possess an invariant probability measure (see \cite[Exercise 29.6]{sato-book}).

We end this paper with the following observations. Regardless the (strong) ergodicity property we have the following limiting behaviors. Let
$\process{L}$ be a $\bar{d}$-dimensional  L\'evy process with Feller generator
$(\mathcal{A}^{\infty},\mathcal{D}_{\mathcal{A}^{\infty}})$ and let $f\in C_c^{\infty}(\R^{\bar{d}})$. Note that $\mathcal{A}^{\infty}f\in L^{1}(dx,\R^{\bar{d}})\cap B_b(\R^{\bar{d}})$. Then, since
$\hat{f}\in L^{1}(dx,\R^{\bar{d}})$ (recall that $\hat{f}(\xi)$ denotes the Fourier transform of $f(x)$), we have
$$f(x)=\int_{\R^{d}}e^{i\langle\xi,x\rangle}\hat{f}(\xi)d\xi$$ and by
using an analogous   approach as in the proof of Theorem \ref{tm1.3}, we obtain that
$$n^{-\beta}\int_0^{nt}\mathcal{A}^{\infty}f(L_s)ds\stackrel{\hbox{\scriptsize{$\textrm{d}$}}}{\longrightarrow}0$$
for any  $\beta>0$ and any initial distribution $\rho(dx)$ of $\process{L}$. Further, if $\process{F}$  is a
$\bar{d}$-dimensional diffusion with jumps with  Feller generator
$(\mathcal{A}^{\infty},\mathcal{D}_{\mathcal{A}^{\infty}})$ and symbol $q(x,\xi)$ satisfying \begin{align}\label{eq3.1}\int_{\{|\xi|<r\}}\frac{d\xi}{\inf_{x\in\R^{\bar{d}}}{\rm Re}\, q(x,\xi)}<\infty,\end{align} for some $r>0$, and the conditions in \eqref{eq3.02} and (\textbf{C6}). Then, according to \cite[the proof of Theorem 1.1]{rene-wang-feller},
$$n^{-\beta}\int_0^{nt}f(F_s)ds\stackrel{\hbox{\scriptsize{$\mathbb{P}^{\rho}\textrm{-}\, \textrm{a.s.}$}}}{\xrightarrow{\hspace*{1cm}}}0$$
for any $\beta>0$, any $f\in C_c^{\infty}(\R^{\bar{d}})$ and any initial distribution $\rho(dx)$ of $\process{F}$.  Let us remark that the condition in (\ref{eq3.1}), together with  \eqref{eq3.02} and (\textbf{C6}), implies the transience of $\process{F}$ (see \cite[Theorem 1.1]{rene-wang-feller}). In the L\'evy process case, by the definition of transience, assumptions  \eqref{eq3.02} and (\textbf{C6}) are not needed (see \cite[Theorem 37.5]{sato-book}).
 \begin{proposition}Let $\process{F}$ be a $\bar{d}$-dimensional diffusion with jumps with symbol $q(x,\xi)$ satisfying \eqref{eq3.02} and \begin{align}\label{eq3.2}\int_{0}^{\infty}\int_{\R^{\bar{d}}}\exp\left[-t\inf_{x\in\R^{\bar{d}}}{\rm Re}\, q(x,\xi)\right]d\xi dt<\infty.\end{align} Then,
 $$n^{-\beta}\int_0^{nt}f(F_s)ds\stackrel{\hbox{\scriptsize{$\mathbb{P}^{\rho}\textrm{-}\, \textrm{a.s.}$}}}{\xrightarrow{\hspace*{1cm}}}0$$
for any $\beta>0$, any $f\in L^{1}(dx,\R^{\bar{d}})$ and any   initial distribution $\rho(dx)$ of $\process{F}$. In addition, if $f\in L^{1}(dx,\R^{\bar{d}})\cap B_b(\R^{\bar{d}}),$ then it suffices to require that \eqref{eq3.2} holds on  $(t_0,\infty)$, for some $t_0>0$.
\end{proposition}
\begin{proof}
First, recall that, according to \cite[Theorem 2.6]{sandric-tams} and \cite[Theorem 1.1]{rene-wang-feller}, $\process{F}$ has a transition density function $p(t,x,y)$ which satisfies $$\sup_{x,y\in\R^{\bar{d}}}p(t,x,y)\leq (4\pi)^{-\bar{d}}\int_{\R^{\bar{d}}}\exp\left[-\frac{t}{16}\inf_{x\in\R^{\bar{d}}}{\rm Re}\,q(x,\xi)\right]d\xi.$$ By using this fact, for any $f\in L^{1}(dx,\R^{\bar{d}})$ and any initial distribution $\rho(dx)$ of $\process{F}$,   we have
\begin{align*}\mathbb{E}^{\rho}\left[\int_0^{\infty}|f(F_s)|ds\right]&=\int_{\R^{\bar{d}}}\int_0^{\infty}\int_{\R^{\bar{d}}}|f(y)|p(s,x,y)dyds\rho(dx)\\&\leq (4\pi)^{-\bar{d}}\int_{\R^{\bar{d}}}|f(y)|dy\int_0^{\infty}\int_{\R^{\bar{d}}}\exp\left[-\frac{s}{16}\inf_{x\in\R^{\bar{d}}} {\rm Re}\, q(x,\xi)\right]d\xi ds.\end{align*} In particular,
$$\int_0^{\infty}f(F_s)ds<\infty,\quad \mathbb{P}^{\rho}\textrm{-a.s.},$$ which proves the claim.

To prove the second statement, by completely the same reasoning as above we  conclude that
$$\int_{t_0}^{\infty}f(F_s)ds<\infty,\quad \mathbb{P}^{\rho}\textrm{-a.s.},$$ for some $t_0>0$. Finally, due to boundedness of $f(x)$,
$$\int_{0}^{\infty}f(F_s)ds<\infty,\quad \mathbb{P}^{\rho}\textrm{-a.s.},$$ which again implies the desired result.
\end{proof}
Finally, let $\process{M}$ be a $\bar{d}$-dimensional progressively measurable Markov process possessing the \emph{local-time} process (occupation measure), that is, a nonnegative process $\{l(t,y)\}_{t\geq0}$, such that for any $x\in\R^{\bar{d}}$, any $t\geq0$ and any nonnegative $f\in B_b(\R^{\bar{d}})$,  $$\int_0^{t}f(M_s)ds=\int_{\R^{\bar{d}}}f(y)l(t,y)dy,\quad \mathbb{P}^{x}\textrm{-a.s.}$$
A sufficient condition for the existence of the local time for a diffusion with jumps $\process{F}$ with symbol $q(x,\xi)$ satisfying \eqref{eq3.02} is as follows
$$ \int_{\R^{\bar{d}}}\frac{d\xi}{1+\inf_{x\in\R^{\bar{d}}}{\rm Re}\,q(x,\xi)}<\infty$$ (see \cite[Theorem 1.1]{rene-wang-feller}).
Now, let $f\in L^{1}(dx,\R^{\bar{d}})$. Then, for any $\beta>0$ and any $t>0$, we have $$n^{\beta \bar{d}}\int_0^{t}f(n^{\beta}M_s)ds=n^{\beta \bar{d}}\int_{\R^{\bar{d}}}f(n^{\beta}y)l(t,y)dy=\int_{\R^{\bar{d}}}f(y)l(t,n^{-\beta}y)dy.$$ Hence, if
$\{l(t,y)\}_{t\geq0}$ is continuous in $y$, $\mathbb{P}^{x}\textrm{-a.s.}$ for all $x\in\R^{\bar{d}}$, we have
$$n^{\beta \bar{d}}\int_0^{t}f(n^{\beta}M_s)ds\stackrel{\hbox{\scriptsize{$\mathbb{P}^{\rho}\textrm{-}\, \textrm{a.s.}$}}}{\xrightarrow{\hspace*{1cm}}}l(t,0)\int_{\R^{\bar{d}}}f(y)dy$$
for any $\beta>0$, any $f\in L^{1}(dx,\R^{\bar{d}})$ and any initial distribution $\rho(dx)$ of $\process{L}$.
In the one-dimensional L\'evy process case necessary and sufficient conditions for the existence and continuity of local times have been given in  \cite{barlow-local-1}, \cite{barlow-local-2} and \cite{bertoin-book}.

 \section*{Acknowledgement}
Financial support through the Marcus Endowment at Penn State (for Guodong Pang), the Croatian Science Foundation (under Project 3526), NEWFELPRO Programme (under Project 31) and  Dresden Fellowship Programme (for Nikola Sandri\'c) is gratefully acknowledged.
The authors thank Alexei Novikov at Department of Mathematics in Penn State for showing us his work \cite{kom-nov-ryz} and helpful discussions, which have motivated this work.
The authors also thank the anonymous reviewer for careful reading of the paper and
for helpful comments that led to improvement in the presentation.

\bibliographystyle{alpha}
\bibliography{References}

\end{document}